\documentclass[11pt,twoside]{article}
\usepackage[utf8]{inputenc}
\usepackage{amsmath,amssymb,latexsym}
\usepackage{amsthm}
\usepackage{a4wide}
\usepackage{fourier}
\usepackage{array}
\usepackage{stackengine}
\usepackage{faktor}
\usepackage{fancyhdr}
\usepackage{graphicx}
\usepackage[svgnames]{xcolor}
\usepackage[labelfont=bf,width=0.8\textwidth,justification=centering]{caption}
\usepackage{tikz}
\usepackage{tikz-cd}
\usepackage{subfigure}
\usepackage[colorlinks=true,linkcolor=Indigo,citecolor=Indigo]{hyperref
 }
  \makeatletter
 \newenvironment{sqcases}{%
  \matrix@check\sqcases\env@sqcases
}{%
  \endarray\right.%
}
\def\env@sqcases{%
  \let\@ifnextchar\new@ifnextchar
  \left\lbrack
  \def\arraystretch{1.2}%
  \array{@{}l@{\quad}l@{}}%
}
\makeatother
\tikzset{every picture/.style={line width=0.75pt}}
 \usetikzlibrary{patterns}
 \newcommand{\dd}[2]{\mathrm{d} #1\wedge\mathrm{d} #2}
  \newcommand{\partt}[2]{\frac{\partial #2}{\partial #1} \ }
\newtheorem{theorem}{Theorem}[section]
\newtheorem{proposition}[theorem]{Proposition}
\newtheorem{definition}[theorem]{Definition}
\newtheorem{lemma}[theorem]{Lemma}
\theoremstyle{definition}

\newtheorem{rmk}[theorem]{Remark}
\numberwithin{equation}{section}
\numberwithin{figure}{section}

\title{A note on the geometry of the  two-body problem on $S^2$}
\author{Alessandro Arsie \ \& \ Nataliya A. Balabanova }
\date{June 2023}

\numberwithin{equation}{section}
\numberwithin{figure}{section}

\begin{document}
\maketitle 
  \begin{abstract}
  Leveraging on the results of \cite{arsie2023collision}, we carry out an investigation of the algebraic three-fold $\Sigma_{C,h}$, the common level set of the Hamiltonian and the Casimir, for the two-body problem for equal masses on $S^2$ subject to a gravitational potential of cotangent type. We determine the topology of its compactification $\overline{\Sigma}_{C,h}$ and how it bifurcates with respect to the admissible values of $(C,h)$, ($C$ being the fixed value of the Casimir and $h$ the fixed value of the Hamiltonian). This bifurcation diagram is actually equal to the bifurcation diagram that describes relative equilibria. 
  We also prove that for $h$ sufficiently negative $\Sigma_{C,h}$ is equipped with a global contact form obtained from the environment symplectic form via a suitable Liouville vector field. 
\medskip


\end{abstract}
 \tableofcontents 
\section{Introduction}

The history of the interactions between Geometry and Dynamics is long and distinguished (\cite{fomenko2013algebra, ratiu2005crash}). On one hand Geometry provides a global framework to analyse dynamical behaviour, while on the other Dynamics has led to the discovery of new geometrical and topological structures (Symplectic/Poisson/Contact geometry and their topological aspects, just to name a few). 

In recent years, there has been an increased interests in using tools from Symplectic and Contact Topology to study the global structure of complicated dynamical systems. Two recent beautifully written expositions about this area are \cite{frauenfelder2018restricted} and \cite{moreno2022contact}. 
One of the main tools in this study is the realisation that often a Hamiltonian vector field on a fixed level set of the Hamiltonian is simply a positive time reparameterization of the Reeb vector field of a contact manifold (for more details see Section \ref{contact_section}). For instance, this allows one to use results coming from Contact Topology pertaining the existence of closed trajectories for the Reeb vector field to conclude for the existence of periodic solutions of a Hamiltonian system. These techniques are definitely not perturbative, so they allow one to explore regimes that are beyond the reach of tools like bifurcation theory. 
For example, given a Riemannian manifold $(M,g)$, and a mechanical Hamiltonian $H:T^*M \rightarrow \mathbb{R}$ where locally in cotagent coordinates $H(q,p):=\frac{1}{2}|p|_g+V(q)$, $|\cdot|_g$ is the norm induced by the inverse of $g$ and $V$ is smooth function on $M$, it is well known that if $c>\sup(V)$, then the energy hypersurface $\Sigma:=H^{-1}(c)$ is fibre-wise star-shaped and thus of contact type (see \cite{frauenfelder2018restricted}, Remark 2.6.5).

Our goal in this note follows this line of development, although is much more limited . We study the topology of the compactification of the common level set of a Hamiltonian and a Casimir for a Hamiltonian system that describes the interaction of two equal masses on $S^2$ interacting through a gravitation potential of cotangent type, (the generalisation of the gravitational potential on a surface of positive constant curvature). 

Using an ingenious symplectic reduction developed in \cite{borisov2018reduction}, the problem is reduced from eight dimensions to a five-dimensional Poisson manifold. Inside this five-dimensional Poisson manifold, we consider the compactified three-fold $\overline{\Sigma}_{C,h}$ obtained by taking the common level set of the (reduced) Hamiltonian $\mathcal{H} = h$ and the non-trivial Casimir $\mathcal{C} = C$.
We analyse the homeomorphism type of this three-fold and how it depends on the values of $h$ and $C$. 
After having recalled some basic facts about Contact Geometry, we show that $\Sigma_{C,h}$ is a hypersurface of contact type if $h$ is sufficiently negative.
In this case, the Hamiltonian vector field is a suitable positive time reparameterisation of the Reeb vector field. Unfortunately, since $\Sigma_{C,h}$ is not compact, we can not use the corresponding theorems about the existence of a closed characteristic for the Reeb vector field.

In general, it is a difficult problem to determine if a (regularised) energy hypersurface admits a compatible contact form. 
For instance, it has been proved in \cite{albers2012contact}
the the regularised energy hypersurface of the planar restricted three-body problem if of restricted contact-type for all energies below the one corresponding to the first Lagrange point and for those slightly above that. See also \cite{moreno2022contact} for other examples. 

This note was prompted by the desire to investigate some geometric aspects left unexplored in the work 
 \cite{arsie2023collision}, where the authors analysed collision trajectories and the regularization for the two-body problem with equal masses on $S^2$ subject precisely to a gravitational potential of cotangent type.

%In this paper, we consider the two body problem for two equal masses on $S^2$ subject to a ``gravitational" potential of cotagent type (the generalization of the gravitational potential on a surface of positive constant curvature). Using an ingenious symplectic reduction developed in \cite{borisov2018reduction}, we reduce the problem from eight dimensions to a five-dimensional Poisson manifold. Inside this manifold we consider the regularized and compactified three-fold $\Sigma_{C,h}$ obtained by taking the common level set of the (reduced) Hamiltonian $\mathcal{H} = h$ and the non-trivial Casimir $\mathcal{C} = c$.

\section{Preliminaries}

We just sketch the setup of the problem, directing the reader towards \cite{arsie2023collision, borisov2018reduction,garcia2021attracting}, for more details.  

We assume that the two bodies of equal mass (that can be without loss of generality assumed to be equal to 1) are placed on the surface of a two-dimensional unit sphere $S^2$ and are interacting with an attracting potential. 

The in-built $SO(3)$-symmetry of the setup allows us to perform reduction with respect to the symplectic $SO(3)$ action on $T^{\ast}S^2\times T^{\ast}S^2$, leaving us with a five-dimensional system in the variables $q$, $p$, $m_1$, $m_2$, $m_3$. The potential $V(q)$ is assumed to be $-\cot(q)$. 

The reduced variables have the following physical meaning: $q$ is the angle that separates the two bodies, $p$ its Lagrangian dual, and $m_1, \ m_2, \ m_3$ are the coordinates of the angular momentum in  the coordinate system that moves with the two bodies. For convenience, we sometimes make the substitution $\xi = \cot(q)$

The system has two invariant quantities: the Hamiltonian 
\begin{equation}\label{Ham.eq1}
\mathcal{H} = \frac12\left(m_1^2 + m_2^2  - 2m_1p + 2p^2 + \xi(-2-2m_2m_3 + m_3^2\xi) + m_3^2(1 + \xi^2)\right)
\end{equation}
and the Casimir 
\[
\mathcal{C}= m_1^2 + m_2^2 + m_3^2.
\]

The symplectic structure on $T^{\ast}S^2\times T^{\ast}S^2$ reduces to a Poisson structure on $\mathbb{R}^{2}\times \mathbb{R}^3$, with   non-zero Poisson brackets in the reduced variables given by 
    \begin{align}
    \label{eq: Poiss}
        &\{m_1,m_2\} = -m_3,\ &  &\{m_2,m_3\} = -m_1 ,\nonumber\\ &\{m_1,m_3\} = m_2, \ &   \   &\{\xi,p\} = -(\xi^2+1).
    \end{align}

\section{Topology of the compactification of $\Sigma_{C,h}$}
\label{sec: topology}
\begin{figure}
\centering
  \subfigure[$C=6.02, \ h = 2.7$]{\includegraphics[scale=.3]{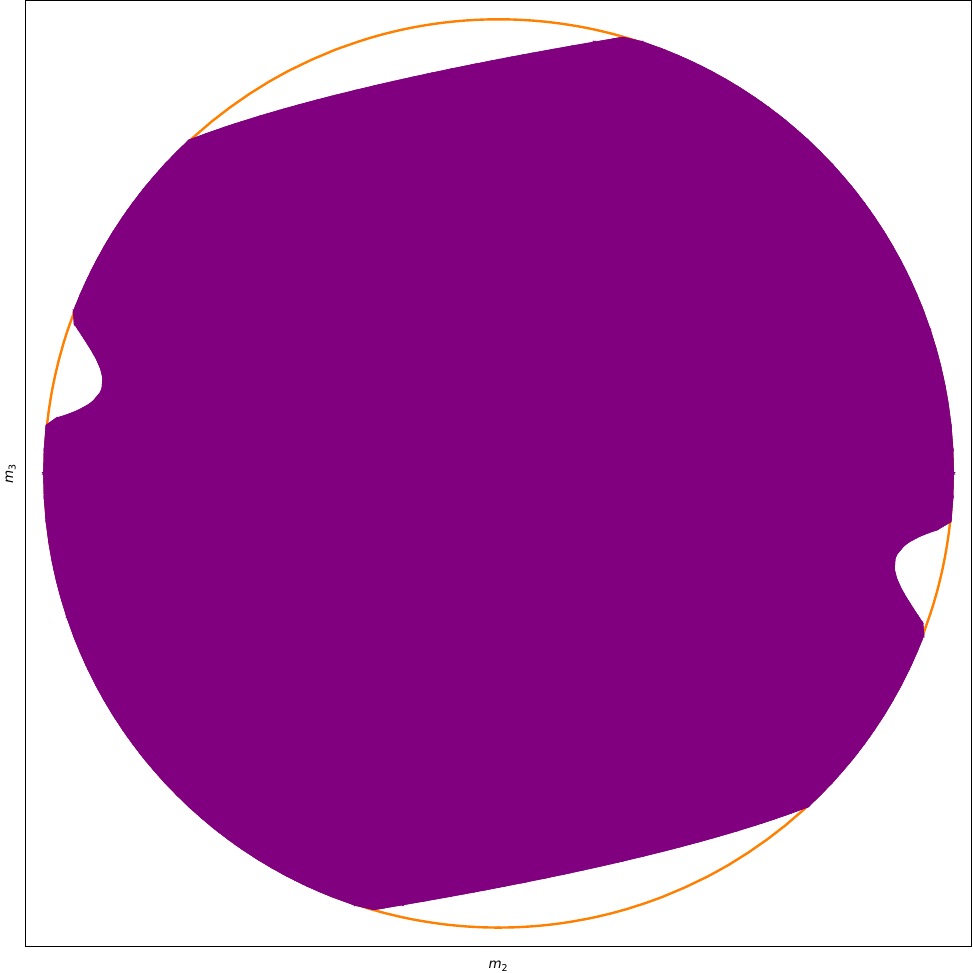}}
  \subfigure[$C=6.07, h=2.25$]{\includegraphics[scale=.3]{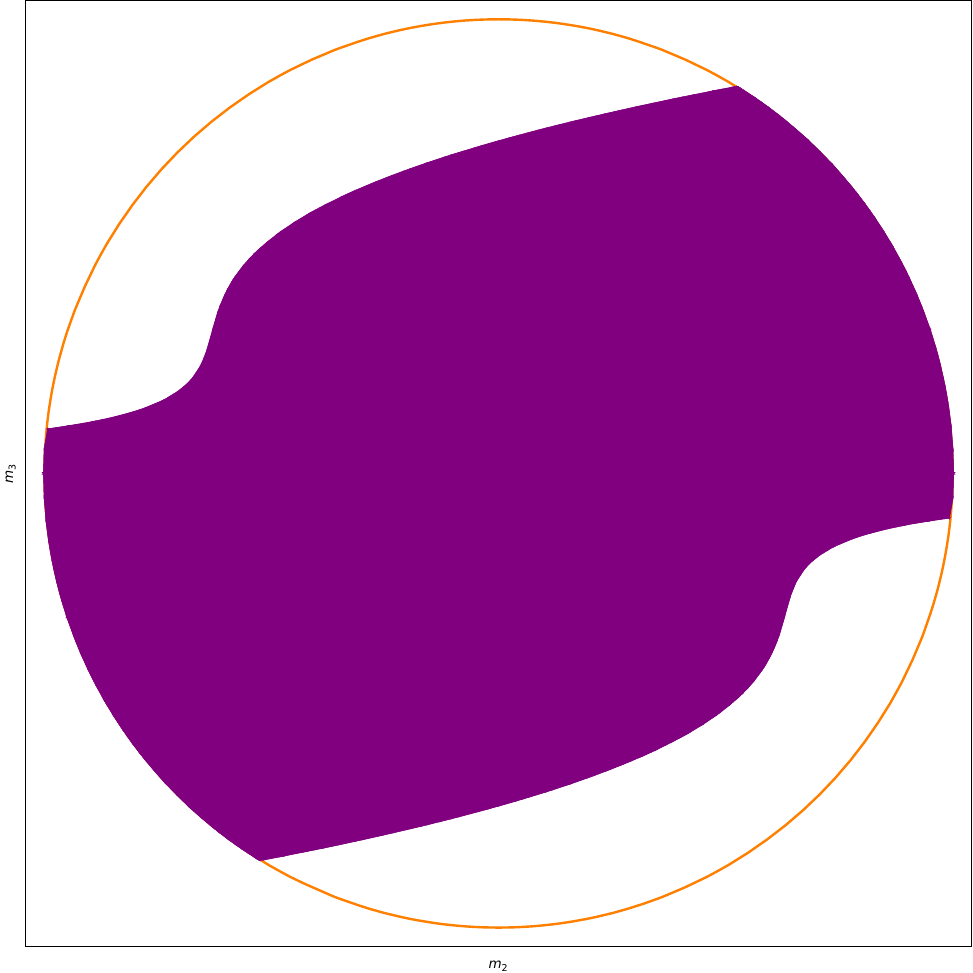}}
  \subfigure[$C=1, h=20$]{\includegraphics[scale=.3]{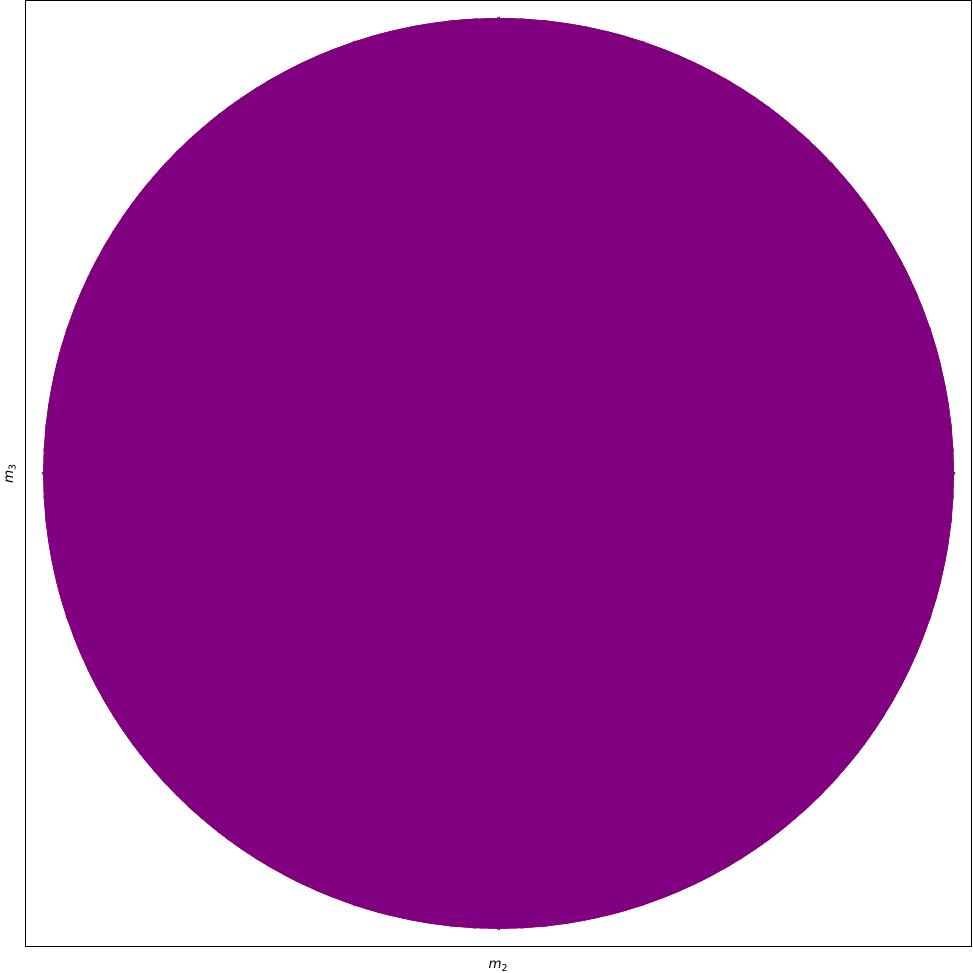}}
  \caption{Projection of $\pi(\Sigma_c)$ onto $(m_2,m_3)$-plane}
  \label{fig: sigmac}
\end{figure}

In this Section we study the compactification of the invariant variety given by the intersection of the energy hypersurface $\Sigma_h$ and the Casimir hypersurface $\Sigma_C$ in $\mathbb{R}^5$. We show that topologically this is either $S^1\times S^2$ or the connected sum of three copies of $S^1\times S^2.$ This analysis is partly inspired by the ideas detailed in \cite{bolsinov2004integrable}. One important difference however, is that the system we are considering is not integrable in general (see \cite{shchepetilov2006}), therefore it does not give rise to a two-dimensional Lagrangian foliation.

The energy hypersurface $\Sigma_h$ is given by:
\[\Sigma_h:=\left\{( \xi, p,m_1, m_2, m_3) \; |\;  \frac12\left(m_1^2 + m_2^2  - 2m_1p + 2p^2 + \xi(-2-2m_2m_3 + m_3^2\xi) + m_3^2(1 + \xi^2)\right)=h \right\},\]
while the Casimir hypersurface $\Sigma_C$ is given by:
$$\Sigma_C:=\{(m_1, m_2, m_3, \xi, p)\in \mathbb{R}^5, \text{ such that } m_1^2+m_2^2+m_3^2=C\}.$$
Notice that $\Sigma_C$ is non-empty only for $C\geq 0$ and for $C>0$ it is a four-dimensional cylinder in $\mathbb{R}^5$ of the form $S^2_{C}\times \mathbb{R}^2$, where $(\xi,p)\in \mathbb{R}^2$ and $S^2_{C}$ is the two-dimensional sphere of radius $\sqrt{C}$ in $\mathbb{R}^3$ (the space of $(m_1,m_2,m_3)$). Of course $S^2_{C}$ reduces to a point for $C=0.$ 

For the moment we consider the case $C>0.$
The invariant three-dimensional variety on which the dynamics unfolds is given by $\Sigma_{C,h}:=\Sigma_{C}\cap \Sigma_h\subset \mathbb{R}^5.$
In this case, there is a natural projection $\pi: \Sigma_C \rightarrow S^2_C$. Restricting $\pi$ to $\Sigma_{C,h}\subset \Sigma_C$, we obtain a projection which we still denote with $\pi: \Sigma_{C,h}\rightarrow S^2_C.$ As we shall see below, the image of this projection is is either $S^2_C$, or $S^2_C\setminus\{\Delta_1, \Delta_2\}$ or $S^2_C\setminus \cup_{i=1}^4 \Delta_i$, where $\Delta_i$s are open disks. In order to determine the topology of the (compactification) of $\Sigma_{C,h}$ we study the inverse image $\pi^{-1}(P)$ as $P$ varies in $S^2_C$.

Consider the equation for the   Hamiltonian (\ref{Ham.eq1})  rewritten as sum of two squares using the Casimir:
\begin{equation}
\label{eq:full squares}
2\left(p - \frac{m_1}{2}\right)^2  + 2\left(m_3\xi-\frac{m_2m_3 +1}{2m_3
}\right)^2 = 2h-C + \frac{m_1^2}{2} + 2\left(\frac{m_2m_3+1}{2m_3
}\right)^2 
\end{equation}

Using (\ref{eq:full squares}), we prove:
\begin{lemma}\label{lemma:projection}
If a point $P\in S^2_C$ belongs to the image of the projection $\pi: \Sigma_{C,h}\rightarrow S^2_C$, its pre-image in $\Sigma_{C,h}$ is one of the following:
    \begin{enumerate}
        \item an ellipse, if $P$ lies strictly inside $\mathrm{Im}\ \pi$ and not on the equator $m_3=0$;
        \item a parabola, if $P$ belongs to the equator $m_3=0$;
        \item a point, if $P$ belongs to the boundary of $\mathrm{Im}\ \pi$.
    \end{enumerate}
\end{lemma}
\begin{proof}
First we prove that the circle $\{m_3=0\}\subset S^2_C$ is always in the image of $\pi.$ Indeed, set $m_3=0$ and consider a point $P=(m_1, m_2, 0)\in S^2_C$. Then using \eqref{Ham.eq1} and the Casimir, $\pi^{-1}(P)$ is described by the parabola \begin{equation}\mathcal{C}_{C,h,m_1}:=\{(p ,\xi)\in \mathbb{R}^2 \text{ such that }C-2m_1p+2p^2-2\xi=2h\}\end{equation} inside $\mathbb{R}^2$ with coordinates $(\xi, p),$ while $(C,h,m_1)$ are all fixed. Incidentally, these are precisely the fibres in the projection $\pi$ that cause $\Sigma_{C,h}$ to be non-compact. 
If $P$ is strictly inside the image of the projection $\pi$ and not on the equator $\{m_3=0\},$ then the right hand side of \eqref{eq:full squares} is strictly positive and therefore $\pi^{-1}(P)$ is indeed an ellipse in the plane with coordinates $(\xi, p).$
If a point $P$ belongs to the boundary of the image of $\pi$ then the right hand side of \eqref{eq:full squares} is identically zero and this corresponds to a unique point in the plane with coordinates ($\xi, p$). Finally, if the right hand side of \eqref{eq:full squares} is strictly negative, then $P$ does not belong to the image of $\pi.$
\end{proof}
\begin{figure}
   \centering

\tikzset{every picture/.style={line width=0.75pt}} %set default line width to 0.75pt        

\begin{tikzpicture}[x=0.75pt,y=0.75pt,yscale=-1,xscale=1]
%uncomment if require: \path (0,327); %set diagram left start at 0, and has height of 327

%Shape: Ellipse [id:dp8731883172650103] 
\draw   (109,31.5) .. controls (109,24.6) and (129.59,19) .. (155,19) .. controls (180.41,19) and (201,24.6) .. (201,31.5) .. controls (201,38.4) and (180.41,44) .. (155,44) .. controls (129.59,44) and (109,38.4) .. (109,31.5) -- cycle ;
%Shape: Ellipse [id:dp5512685815144243] 
\draw   (108,121.5) .. controls (108,114.6) and (128.59,109) .. (154,109) .. controls (179.41,109) and (200,114.6) .. (200,121.5) .. controls (200,128.4) and (179.41,134) .. (154,134) .. controls (128.59,134) and (108,128.4) .. (108,121.5) -- cycle ;
%Straight Lines [id:da1675974186620628] 
\draw    (109,31.5) -- (108,121.5) ;
%Straight Lines [id:da6509224925735195] 
\draw    (201,31.5) -- (200,121.5) ;
%Straight Lines [id:da9533760104276048] 
\draw [line width=3]    (174,42.5) -- (173,132.5) ;
%Straight Lines [id:da5396653748209892] 
\draw    (282,80) -- (218,80.08) ;
\draw [shift={(216,80.08)}, rotate = 359.93] [color={rgb, 255:red, 0; green, 0; blue, 0 }  ][line width=0.75]    (10.93,-3.29) .. controls (6.95,-1.4) and (3.31,-0.3) .. (0,0) .. controls (3.31,0.3) and (6.95,1.4) .. (10.93,3.29)   ;
%Straight Lines [id:da8615900491170272] 
\draw    (332,31.5) -- (331,121.5) ;
%Shape: Circle [id:dp3959815185173534] 
\draw  [fill={rgb, 255:red, 0; green, 0; blue, 0 }  ,fill opacity=1 ] (328.5,31.5) .. controls (328.5,29.57) and (330.07,28) .. (332,28) .. controls (333.93,28) and (335.5,29.57) .. (335.5,31.5) .. controls (335.5,33.43) and (333.93,35) .. (332,35) .. controls (330.07,35) and (328.5,33.43) .. (328.5,31.5) -- cycle ;
%Shape: Circle [id:dp8093030164829962] 
\draw  [fill={rgb, 255:red, 0; green, 0; blue, 0 }  ,fill opacity=1 ] (327.5,121.5) .. controls (327.5,119.57) and (329.07,118) .. (331,118) .. controls (332.93,118) and (334.5,119.57) .. (334.5,121.5) .. controls (334.5,123.43) and (332.93,125) .. (331,125) .. controls (329.07,125) and (327.5,123.43) .. (327.5,121.5) -- cycle ;
%Shape: Ellipse [id:dp4060466695406666] 
\draw   (310,46.67) .. controls (310,44.09) and (319.4,42) .. (331,42) .. controls (342.6,42) and (352,44.09) .. (352,46.67) .. controls (352,49.24) and (342.6,51.33) .. (331,51.33) .. controls (319.4,51.33) and (310,49.24) .. (310,46.67) -- cycle ;
%Shape: Ellipse [id:dp6072484996986744] 
\draw   (311,106.33) .. controls (311,103.76) and (320.4,101.67) .. (332,101.67) .. controls (343.6,101.67) and (353,103.76) .. (353,106.33) .. controls (353,108.91) and (343.6,111) .. (332,111) .. controls (320.4,111) and (311,108.91) .. (311,106.33) -- cycle ;
%Shape: Ellipse [id:dp7477368165725964] 
\draw   (292,64) .. controls (292,60.32) and (309.46,57.33) .. (331,57.33) .. controls (352.54,57.33) and (370,60.32) .. (370,64) .. controls (370,67.68) and (352.54,70.67) .. (331,70.67) .. controls (309.46,70.67) and (292,67.68) .. (292,64) -- cycle ;
%Shape: Ellipse [id:dp6146317354600495] 
\draw   (292.5,83.17) .. controls (292.5,79.48) and (309.96,76.5) .. (331.5,76.5) .. controls (353.04,76.5) and (370.5,79.48) .. (370.5,83.17) .. controls (370.5,86.85) and (353.04,89.83) .. (331.5,89.83) .. controls (309.96,89.83) and (292.5,86.85) .. (292.5,83.17) -- cycle ;

% Text Node
\draw (242,55.75) node [anchor=north west][inner sep=0.75pt]   [align=left] {$\displaystyle \pi $};

\end{tikzpicture}

    \caption{When  $\pi(\overline{\Sigma}_{C,h})$ is a sphere with two holes (i,e, a cylinder), the preimage of the every 'vertical' interval on this cylinder is a sphere, entailing $\overline{\Sigma}_{C,h}\simeq S^1\times S^2$.}
  \label{fig:preimage pi}
\end{figure}
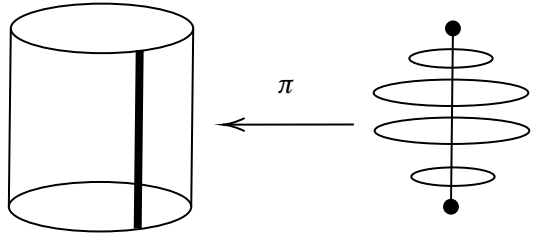

As it was stated in the proof of Lemma \ref{lemma:projection}, the preimage via $\pi$ of the circle $\{m_3 = 0\}\subset S^2_C$  is responsible for the non-compactness of $\Sigma_{C,h}$.

However, that can be easily remedied by adding one point to each parabola and thus making them diffeomorphic to circles. In this way, we get a new compactified variety $\overline{\Sigma}_{C,h}$, which we proceed to project onto $S^2_C$ in the same manner; now the preimage of each point on the sphere is either empty, or a point or a curve diffeomorphic to a circle.

\begin{proposition}\label{prop3.2}
The compactification $\overline{\Sigma}_{C,h}$ is a hypersurface in $S^2\times S_C^2.$
\end{proposition}
\begin{proof}
To formalise what we have said above, we apply the inverse stereographic projection to the $(\xi,p)$ plane: the North pole will serve as the 'additional' point we glue to the parabolae in order to close them. We substitute $\xi  =\frac{x}{1-z}, \ p = \frac{y}{1-z}$ and multiply $\mathcal{H}$ by $(1-z)^2$. In doing so, we describe $\overline{\Sigma}_{C,h}$, the compactification of $\Sigma_{C,h}$, as the common level set of the three polynomials in $\mathbb{R}^6$. After a straightforward computation, these are given by:
\begin{equation}\label{3poly.eq}
\overline{\Sigma}_{C,h} =\begin{cases}m_1^2 + m_2^2 + m_3^2 = C, \\x^2 + y^2 + z^2 = 1,\\ (m_1^2 + m_2^2 + m_3^2-2h)(1-z)^2 + 2y^2 - 2m_1y(1-z) + 2m_3^2x^2 - 2x(1-z)(1+m_2m_3) = 0\end{cases},
\end{equation}
which is indeed a hypersurface in $S^2\times S_C^2$. 
\end{proof}

In \cite{borisov2018reduction}, the authors present a diagram in the $(C^2,h)$-plane (see Figure 9 (b) in the aforementioned work) the describes the bifurcation of relative equilibria (in our case, equilibria in the reduced coordinates $(m_1, m_2, m_3, \xi, p)$) according to different values of $h$ and $C^2$. In particular, the curves in that diagram describe the pairs $(C^2,h)$ for which relative equilibria exist. Outside those curves, relative equilibria do not exist. For convenience we call the union of all those curves $\gamma$. In Subsection \ref{bif.sub1} we are going to prove that this bifurcation diagram is also the bifurcation diagram that controls the change in the topology of $\overline{\Sigma}_{C,h}.$ 

\begin{lemma}\label{lemma.smooth}
    The variety $\Sigma_{C,h}$ is smooth whenever $(C^2,h)\notin \gamma.$  
\end{lemma}
\begin{proof}
 We need to show that the two differentials $\mathrm{d}\mathcal{H}$ and $\mathrm{d}\mathcal{C}$ are linearly independent on $\Sigma_h\cap\Sigma_C.$

First, if $(C^2,h)\notin \gamma$, then $d\mathcal{H}$ is never zero on $\Sigma_{C,h},$ since if it were zero somewhere, it would give rise to relative equilibria.  

Note that the Poisson structure $\sigma$ (see equation \ref{eq: Poiss})) has rank 4 everywhere on the points of $\Sigma_{C,h}$, since in particular it has rank 4 on $\Sigma_C.$ Therefore its kernel is one-dimensional and it is spanned by $d\mathcal{C}$
on the points of $\Sigma_{C,h},$ since there $d\mathcal{C}$ is always non-vanishing there. %Additionally, it can be demonstrated through a simple computation  that $d\mathcal{H}$ does not vanish unless $C=0$. 

Therefore, both $d\mathcal{H}$ and $d\mathcal{C}$ are nowhere vanishing on $\Sigma_{C,h}$. Thus they are linearly dependent iff they are proportional via a non-where zero factor. On the other hand if $d\mathcal{H}$ were proportional to $d\mathcal{C}$ at a point of $\Sigma_{C,h}$, then the corresponding Hamiltonian vector field would have an equilibrium there, while we are considering the pairs $(C,h)$ for which relative equilibria do not exist (i.e. equilibria for the Hamiltonian system in $(m_1 m_2, m_3, \xi, p)).$ 
This proves that $\Sigma_{C,h}$ is indeed smooth if the pair $(C^2,h)\notin \gamma.$ 
\end{proof}

\begin{lemma} Suppose the pair $(C^2,h)\notin \gamma$. Then $\mathrm{Sing}(\overline{\Sigma}_{C,h})=\overline{\Sigma}_{C,h}\setminus \Sigma_{C,h},$ where $\mathrm{Sing}$ is the singularity locus.
\end{lemma}
\begin{proof}
By Lemma \ref{lemma.smooth}, it is clear that $\mathrm{Sing}(\overline{\Sigma}_{C,h})\subset\overline{\Sigma}_{C,h}\setminus \Sigma_{C,h}$.
Furthermore, as a variety $\overline{\Sigma}_{C,h}$
fails to be smooth precisely at those points that belong to the common level set \eqref{3poly.eq} for which the Jacobian matrix (or its transpose) of the system \eqref{3poly.eq} fails to have maximum rank. This matrix is given by
\begin{equation}\label{gradient.eq}
\begin{bmatrix}
    2x & 0 &   4\,{{\it m_3}}^{2}x- \left( 2-2\,z \right)  \left( 1+{\it m_2}\,
{\it m_3} \right)  
 \\
    2y & 0 &  4\,y-2\,{\it m_1}\, \left( 1-z
 \right)  
\\
    2z & 0 &  - \left( 2\,C-4\,h \right)  \left( 1-z
 \right) +2\,{\it m_1}\,y+2\,x \left( 1+{\it m_2}\,{\it m_3} \right) 
 
\\
    0 & 2m_1 & -2\,y \left( 1-z \right)\\
    0 & 2m_2 & -2\,x \left( 1-z \right) {\it 
m_3}\,
\\
    0 & 2m_3 &  4\,{\it m_3}\,{x}^{2}-2\,x \left( 1-z \right) {\it 
m_2} 
\end{bmatrix}
\end{equation}
It is clear that the rank of this matrix falls when $z=1$. There are precisely the points that are added to $\Sigma_{C,h}$ in order to make it compact. Outside those points, by Lemma \ref{lemma.smooth} the matrix above has necessarily full rank, provided the pair $(C^2,h)\notin \gamma.$ 
\end{proof}

\begin{proposition}
Let $\pi:\overline{\Sigma}_{C,h}\to S^2_C$ be the usual projection from the compactified variety. Then $\mathrm{Im}\ \pi$ can be:
\begin{enumerate}
    \item a sphere with no holes;
    \item a sphere with two holes;
    \item a sphere with four holes.
\end{enumerate}
\end{proposition}
\begin{proof}
The proof is purely computational. 

 Observe that in order for the preimage of a point $P\in S^2_C$ not to be empty,  the left hand side of (\ref{eq:full squares}) must be non-negative at $P$, i.e.
\begin{equation}
\label{eq: prelim m2m3}
2h-C  + \frac{m_1^2}{2} + 2m_3^2\left(\frac{m_2m_3+1}{2m_3^2}\right)^2\ge 0,
\end{equation}
Since we are considering a subset of $S^2_C$, $m_1^2$ may be replaced by $C-m_2^2-m_3^2$ in the expression above and since \eqref{eq: prelim m2m3} is even in $m_1$ we do not need to take the square root or choose the sign.

Furthermore, once we have substituted $m_1^2=C-m_2^2-m_3^2$ in the expression above, only $m_2$ and $m_3$ appear and effectively we have projected $\pi(\overline{\Sigma}_{C,h})$ to the  $(m_2,m_3)$- plane. The image of this projection will then be contained inside the disk $m_2^2 + m_3^2\le C$. The expression (\ref{eq: prelim m2m3}) turns into 
\[
2h-\frac{C}{2}  -\frac{m_3^2}{2}+ \frac{1}{2m_3^2} + \frac{m_2}{m_3}\ge 0.
\]
The image of $\pi(\overline{\Sigma}_{C,h})$ projected to the $(m_2, m_3)$-plane is clearly the intersection of this region with the disk $m_2^2+m_3^2\le C$. 

The left hand side of  (\ref{eq:full squares}) is even as a function of $m_1$; this means that with the projection to the $(m_2,m_3)$-plane there are no overlaps and no holes are lost; thus, we can deduce the number of holes in $\pi(\overline{\Sigma}_{C,h})$ from the system of equations 
\begin{equation}
\label{eq: sys m2m3}
\begin{cases}
m_2^2 + m_3^2 = C;\\
2h-\frac{C}{2}  -\frac{m_3^2}{2}+ \frac{1}{2m_3^2} + \frac{m_2}{m_3}= 0
\end{cases}
\end{equation}
Expressing $\frac{m_2^2}{m_3^2} = C-1$, substituting it into the second equation, squaring and subsequently multiplying by $4m_3^4$  yield an 8th degree polynomial  in $m_3$:
\begin{equation}
    \label{eq: bif}
\frac{m_3^8}{4}+ \left(\frac{C}{2}-2 h\right)m_3^6 + \left(\frac{C^2}{4}-2 C h+4 h^2+\frac{1}{2}\right)m_3^4 + \left(2 h-\frac{3 C}{2}\right)m_3^2 +\frac{1}{4}=0
\end{equation}

This can be viewed as a  polynomial in $m_3^2$ and thus the corresponding equation can only have an even number of solutions (zero can't be among them). 

To the same end, one can observe that the system (\ref{eq: sys m2m3})  is symmetric with respect to the transformation $(m_2,m_3)\mapsto(-m_2,-m_3)$ and therefore has an even number of solutions, each of which corresponds to an intersection of two curves. Therefore, the number of holes can be 0, 2 or 4, depending on the values of $h$ and $C$. All of these cases take place for varying values of $h$ and $C$ and are depicted  in Figure \ref{fig: sigmac}. 

\begin{figure}
    \centering
\subfigure{

\tikzset{every picture/.style={line width=0.75pt}} %set default line width to 0.75pt        

\begin{tikzpicture}[x=0.75pt,y=0.75pt,yscale=-.7,xscale=.7]
%uncomment if require: \path (0,327); %set diagram left start at 0, and has height of 327

%Shape: Block Arc [id:dp7980656372610546] 
\draw   (470.46,190.02) .. controls (457.7,200.34) and (440.82,206.62) .. (422.31,206.62) .. controls (382.52,206.62) and (350.26,177.6) .. (350.26,141.81) .. controls (350.26,106.01) and (382.52,76.99) .. (422.31,76.99) .. controls (443.61,76.99) and (462.75,85.31) .. (475.94,98.53) -- (458.05,112.96) .. controls (449.09,104.88) and (436.39,99.83) .. (422.31,99.83) .. controls (395.13,99.83) and (373.1,118.62) .. (373.1,141.81) .. controls (373.1,164.99) and (395.13,183.78) .. (422.31,183.78) .. controls (434.48,183.78) and (445.63,180.01) .. (454.22,173.76) -- cycle ;
%Shape: Block Arc [id:dp48159337625234033] 
\draw [-to](20,90) -- (57,110);

\draw (0,70) node [anchor=north west][inner sep=0.75pt]  [font=\normalsize,rotate=0,xslant=-0.03]  {$K$};
\draw   (460.31,228.99) .. controls (439,225.75) and (421.06,210.65) .. (416.34,189.6) .. controls (413.24,175.76) and (416.44,161.98) .. (424.11,150.91) -- (444.61,163.31) .. controls (439.91,169.41) and (437.84,177.08) .. (439.56,184.75) .. controls (442.13,196.2) and (452.39,204.29) .. (464.61,205.75) -- cycle ;
%Shape: Block Arc [id:dp1557203013701467] 
\draw   (424.88,133.3) .. controls (415.42,114.79) and (417.72,92.15) .. (432.54,76.28) .. controls (442.18,65.96) and (455.33,60.43) .. (468.86,59.81) -- (470.77,83.39) .. controls (463.25,83.52) and (455.95,86.43) .. (450.68,92.08) .. controls (442.66,100.66) and (441.64,113.09) .. (447.1,123.27) -- cycle ;
%Shape: Free Drawing [id:dp8735833779064313] 
\draw  [color={rgb, 255:red, 255; green, 255; blue, 255 }  ,draw opacity=1 ][line width=3] [line join = round][line cap = round] (419.1,87.92) .. controls (419.18,88.41) and (419.27,88.91) .. (419.35,89.4) ;
%Shape: Free Drawing [id:dp2438417403363191] 
\draw  [color={rgb, 255:red, 255; green, 255; blue, 255 }  ,draw opacity=1 ][line width=3] [line join = round][line cap = round] (426.48,182.53) .. controls (424.48,184.46) and (420.81,184.48) .. (418.08,183.76) ;
%Shape: Free Drawing [id:dp9335504841039499] 
\draw  [color={rgb, 255:red, 255; green, 255; blue, 255 }  ,draw opacity=1 ][line width=3] [line join = round][line cap = round] (423.43,183.26) .. controls (428.18,183.02) and (432.94,182.77) .. (437.69,182.53) ;
%Shape: Free Drawing [id:dp2620330543826608] 
\draw  [color={rgb, 255:red, 255; green, 255; blue, 255 }  ,draw opacity=1 ][line width=3] [line join = round][line cap = round] (437.44,180.56) .. controls (434.86,187.24) and (430.05,193.59) .. (430.05,200.73) ;
%Shape: Free Drawing [id:dp15235382826595] 
\draw  [color={rgb, 255:red, 255; green, 255; blue, 255 }  ,draw opacity=1 ][line width=3] [line join = round][line cap = round] (430.05,200.97) .. controls (429.71,203.1) and (430.56,205.79) .. (429.03,207.37) .. controls (428.16,208.26) and (426.99,205.89) .. (425.97,205.15) ;
%Shape: Free Drawing [id:dp7877372120222461] 
\draw  [color={rgb, 255:red, 255; green, 255; blue, 255 }  ,draw opacity=1 ][line width=3] [line join = round][line cap = round] (433.62,205.64) .. controls (435.48,205.48) and (439.91,203.47) .. (439.22,205.15) .. controls (438.33,207.29) and (429.78,206.23) .. (432.09,205.64) .. controls (437.75,204.19) and (443.63,203.68) .. (449.41,202.69) ;
%Shape: Free Drawing [id:dp9749297011822264] 
\draw  [color={rgb, 255:red, 255; green, 255; blue, 255 }  ,draw opacity=1 ][line width=3] [line join = round][line cap = round] (415.79,185.97) .. controls (417.82,191.87) and (419.86,197.77) .. (421.9,203.68) ;
%Shape: Free Drawing [id:dp9944140681403211] 
\draw  [color={rgb, 255:red, 255; green, 255; blue, 255 }  ,draw opacity=1 ][line width=3] [line join = round][line cap = round] (415.28,186.95) .. controls (415.87,188.59) and (416.47,190.23) .. (417.06,191.87) ;
%Shape: Free Drawing [id:dp6220401042284267] 
\draw  [color={rgb, 255:red, 255; green, 255; blue, 255 }  ,draw opacity=1 ][line width=3] [line join = round][line cap = round] (440.49,183.51) .. controls (441.76,186.71) and (444.51,189.68) .. (444.31,193.1) .. controls (444.19,195.26) and (438.8,186.27) .. (440.24,187.94) .. controls (443.37,191.57) and (446.35,195.31) .. (449.41,199) ;
%Shape: Free Drawing [id:dp2759887161826957] 
\draw  [color={rgb, 255:red, 255; green, 255; blue, 255 }  ,draw opacity=1 ][line width=3] [line join = round][line cap = round] (422.15,93.01) .. controls (424.45,89.56) and (427.87,86.61) .. (429.03,82.68) .. controls (429.88,79.82) and (421.54,92.08) .. (423.17,89.56) .. controls (425.23,86.39) and (427.93,83.66) .. (430.3,80.71) ;
%Shape: Free Drawing [id:dp8360582662838973] 
\draw  [color={rgb, 255:red, 255; green, 255; blue, 255 }  ,draw opacity=1 ][line width=3] [line join = round][line cap = round] (438.2,77.27) .. controls (427.64,82.36) and (432.65,78.89) .. (455.26,84.15) ;
%Shape: Free Drawing [id:dp3756483721279964] 
\draw  [color={rgb, 255:red, 255; green, 255; blue, 255 }  ,draw opacity=1 ][line width=3] [line join = round][line cap = round] (455.01,84.15) .. controls (449.66,82.51) and (443.94,81.73) .. (438.96,79.23) .. controls (435.18,77.33) and (466.02,79.63) .. (441.76,80.46) ;
%Shape: Free Drawing [id:dp2344114231813128] 
\draw  [color={rgb, 255:red, 255; green, 255; blue, 255 }  ,draw opacity=1 ][line width=3] [line join = round][line cap = round] (455.52,80.71) .. controls (456.2,81.86) and (456.88,83) .. (457.56,84.15) ;
%Shape: Free Drawing [id:dp7841227746475123] 
\draw  [color={rgb, 255:red, 255; green, 255; blue, 255 }  ,draw opacity=1 ][line width=3] [line join = round][line cap = round] (447.88,96.45) .. controls (447.11,98.09) and (444.35,102.71) .. (445.59,101.37) .. controls (448.15,98.58) and (450.37,95.42) .. (451.95,92.02) .. controls (452.47,90.9) and (449,95.46) .. (448.9,94.24) .. controls (448.38,88.28) and (453.47,88.74) .. (457.05,88.09) ;
%Shape: Free Drawing [id:dp567243265702065] 
\draw  [color={rgb, 255:red, 255; green, 255; blue, 255 }  ,draw opacity=1 ][line width=3] [line join = round][line cap = round] (425.97,92.51) .. controls (424.36,93.74) and (422.75,94.97) .. (421.14,96.2) ;
%Shape: Free Drawing [id:dp7464883209093556] 
\draw  [color={rgb, 255:red, 255; green, 255; blue, 255 }  ,draw opacity=1 ][line width=3] [line join = round][line cap = round] (425.72,101.12) .. controls (425.32,100.47) and (424.13,100.29) .. (423.43,100.63) ;
%Shape: Free Drawing [id:dp406750473222663] 
\draw  [color={rgb, 255:red, 255; green, 255; blue, 255 }  ,draw opacity=1 ][line width=3] [line join = round][line cap = round] (423.68,100.63) .. controls (429.37,101.45) and (435.06,102.27) .. (440.75,103.09) ;
%Shape: Free Drawing [id:dp7028602831266684] 
\draw  [color={rgb, 255:red, 255; green, 255; blue, 255 }  ,draw opacity=1 ][line width=3] [line join = round][line cap = round] (440.75,103.09) .. controls (441.09,103.09) and (441.43,103.09) .. (441.76,103.09) ;
%Shape: Free Drawing [id:dp903807668280876] 
\draw  [color={rgb, 255:red, 255; green, 255; blue, 255 }  ,draw opacity=1 ][line width=1.5] [line join = round][line cap = round] (423.68,100.14) .. controls (423.43,98.91) and (423.89,97.28) .. (422.92,96.45) .. controls (422.33,95.95) and (421.69,98.46) .. (421.14,97.92) ;
%Shape: Free Drawing [id:dp9836098772878439] 
\draw  [color={rgb, 255:red, 255; green, 255; blue, 255 }  ,draw opacity=1 ][line width=1.5] [line join = round][line cap = round] (420.37,98.42) .. controls (423.07,98.42) and (422.56,98.42) .. (419.86,98.42) ;
%Shape: Free Drawing [id:dp227361339447264] 
\draw  [color={rgb, 255:red, 255; green, 255; blue, 255 }  ,draw opacity=1 ][line width=1.5] [line join = round][line cap = round] (422.66,98.42) .. controls (422.41,98.74) and (422.15,99.07) .. (421.9,99.4) ;
%Shape: Free Drawing [id:dp8982893396453993] 
\draw  [color={rgb, 255:red, 255; green, 255; blue, 255 }  ,draw opacity=1 ][line width=1.5] [line join = round][line cap = round] (421.39,101.12) .. controls (421.65,100.22) and (421.9,99.32) .. (422.15,98.42) ;
%Shape: Free Drawing [id:dp4282468195959461] 
\draw  [color={rgb, 255:red, 255; green, 255; blue, 255 }  ,draw opacity=1 ][line width=1.5] [line join = round][line cap = round] (444.06,102.35) .. controls (443.29,103.25) and (442.53,104.15) .. (441.76,105.06) ;
%Shape: Free Drawing [id:dp3419332354668869] 
\draw  [color={rgb, 255:red, 255; green, 255; blue, 255 }  ,draw opacity=1 ][line width=1.5] [line join = round][line cap = round] (447.88,103.83) .. controls (446.6,103.5) and (445.33,103.17) .. (444.06,102.84) ;
%Shape: Free Drawing [id:dp031202562284333712] 
\draw  [color={rgb, 255:red, 255; green, 255; blue, 255 }  ,draw opacity=1 ][line width=1.5] [line join = round][line cap = round] (444.31,102.84) .. controls (443.8,103.25) and (443.29,103.66) .. (442.78,104.07) ;
%Shape: Free Drawing [id:dp5983674504493508] 
\draw  [color={rgb, 255:red, 255; green, 255; blue, 255 }  ,draw opacity=1 ][line width=1.5] [line join = round][line cap = round] (429.03,79.48) .. controls (431.42,79.48) and (431.73,79.64) .. (430.05,79.23) ;
%Shape: Free Drawing [id:dp7560172524979984] 
\draw  [color={rgb, 255:red, 255; green, 255; blue, 255 }  ,draw opacity=1 ][line width=1.5] [line join = round][line cap = round] (432.09,79.23) .. controls (431.24,79.23) and (430.39,79.23) .. (429.54,79.23) ;
%Shape: Free Drawing [id:dp49423744070139275] 
\draw  [color={rgb, 255:red, 255; green, 255; blue, 255 }  ,draw opacity=1 ][line width=1.5] [line join = round][line cap = round] (429.54,79.23) .. controls (429.54,79.23) and (429.54,79.23) .. (429.54,79.23) ;
%Shape: Free Drawing [id:dp2570735210389228] 
\draw  [color={rgb, 255:red, 255; green, 255; blue, 255 }  ,draw opacity=1 ][line width=1.5] [line join = round][line cap = round] (437.35,77.43) .. controls (436.16,77.43) and (434.97,77.43) .. (433.78,77.43) ;
%Shape: Free Drawing [id:dp44198016003845875] 
\draw  [color={rgb, 255:red, 255; green, 255; blue, 255 }  ,draw opacity=1 ][line width=1.5] [line join = round][line cap = round] (433.28,78.41) .. controls (433.19,78.25) and (433.11,78.09) .. (433.02,77.92) ;
%Shape: Free Drawing [id:dp4728374825244088] 
\draw  [color={rgb, 255:red, 255; green, 255; blue, 255 }  ,draw opacity=1 ][line width=1.5] [line join = round][line cap = round] (430.47,80.14) .. controls (430.56,79.73) and (430.64,79.32) .. (430.73,78.91) ;
%Shape: Free Drawing [id:dp3115078544078822] 
\draw  [color={rgb, 255:red, 255; green, 255; blue, 255 }  ,draw opacity=1 ][line width=1.5] [line join = round][line cap = round] (456.71,85.55) .. controls (457.13,86.04) and (457.56,86.53) .. (457.98,87.02) ;
%Shape: Free Drawing [id:dp04863717908754328] 
\draw  [color={rgb, 255:red, 255; green, 255; blue, 255 }  ,draw opacity=1 ][line width=1.5] [line join = round][line cap = round] (414.68,184.66) .. controls (415.45,184.66) and (416.21,184.66) .. (416.98,184.66) ;
%Shape: Free Drawing [id:dp20244126772485038] 
\draw  [color={rgb, 255:red, 255; green, 255; blue, 255 }  ,draw opacity=1 ][line width=1.5] [line join = round][line cap = round] (421.56,204.33) .. controls (422.32,204.66) and (423.09,204.99) .. (423.85,205.32) ;
%Shape: Free Drawing [id:dp8405312190541181] 
\draw  [color={rgb, 255:red, 255; green, 255; blue, 255 }  ,draw opacity=1 ][line width=1.5] [line join = round][line cap = round] (426.14,206.3) .. controls (425.97,206.14) and (425.8,205.97) .. (425.64,205.81) ;
%Shape: Free Drawing [id:dp3083988027892002] 
\draw  [color={rgb, 255:red, 255; green, 255; blue, 255 }  ,draw opacity=1 ][line width=1.5] [line join = round][line cap = round] (425.64,206.05) .. controls (425.64,206.22) and (425.64,206.38) .. (425.64,206.55) ;
%Shape: Free Drawing [id:dp35779056058949243] 
\draw  [color={rgb, 255:red, 255; green, 255; blue, 255 }  ,draw opacity=1 ][line width=1.5] [line join = round][line cap = round] (448.3,199.17) .. controls (449.07,199.41) and (449.83,199.66) .. (450.59,199.91) ;
%Shape: Free Drawing [id:dp13324191768923366] 
\draw  [color={rgb, 255:red, 255; green, 255; blue, 255 }  ,draw opacity=1 ][line width=1.5] [line join = round][line cap = round] (449.58,201.13) .. controls (450.08,201.38) and (450.59,201.63) .. (451.1,201.87) ;
%Shape: Free Drawing [id:dp25273242276878216] 
\draw  [color={rgb, 255:red, 255; green, 255; blue, 255 }  ,draw opacity=1 ][line width=1.5] [line join = round][line cap = round] (438.62,181.71) .. controls (439.73,181.71) and (440.83,181.71) .. (441.93,181.71) ;
%Straight Lines [id:da16485873265476259] 
\draw [color={rgb, 255:red, 255; green, 255; blue, 255 }  ,draw opacity=1 ] [dash pattern={on 0.84pt off 2.51pt}]  (423.64,151.6) -- (444.65,163.86) ;
%Straight Lines [id:da31544378133530304] 
\draw [color={rgb, 255:red, 255; green, 255; blue, 255 }  ,draw opacity=1 ] [dash pattern={on 0.84pt off 2.51pt}]  (424.85,133.24) -- (447.1,123.23) ;
%Straight Lines [id:da05041962401574107] 
\draw [color={rgb, 255:red, 255; green, 255; blue, 255 }  ,draw opacity=1 ] [dash pattern={on 0.84pt off 2.51pt}]  (454.62,173.05) -- (471.38,189.26) ;
%Straight Lines [id:da44130625427712866] 
\draw [color={rgb, 255:red, 255; green, 255; blue, 255 }  ,draw opacity=1 ] [dash pattern={on 0.84pt off 2.51pt}]  (464.3,205.61) -- (459.72,228.9) ;
%Straight Lines [id:da9583821935985504] 
\draw [color={rgb, 255:red, 255; green, 255; blue, 255 }  ,draw opacity=1 ] [dash pattern={on 0.84pt off 2.51pt}]  (458.31,113.76) -- (476.77,99.37) ;
%Straight Lines [id:da634661931472791] 
\draw [color={rgb, 255:red, 255; green, 255; blue, 255 }  ,draw opacity=1 ] [dash pattern={on 0.84pt off 2.51pt}]  (470.71,83.44) -- (468.75,59.81) ;

%Shape: Arc [id:dp34393544192605363] 
\draw  [draw opacity=0] (106.59,191.51) .. controls (100.79,193.55) and (94.54,194.67) .. (88.02,194.67) .. controls (57.63,194.67) and (33,170.42) .. (33,140.5) .. controls (33,110.58) and (57.63,86.33) .. (88.02,86.33) .. controls (94.54,86.33) and (100.79,87.45) .. (106.59,89.49) -- (88.02,140.5) -- cycle ; \draw   (106.59,191.51) .. controls (100.79,193.55) and (94.54,194.67) .. (88.02,194.67) .. controls (57.63,194.67) and (33,170.42) .. (33,140.5) .. controls (33,110.58) and (57.63,86.33) .. (88.02,86.33) .. controls (94.54,86.33) and (100.79,87.45) .. (106.59,89.49) ;  
%Shape: Arc [id:dp4845197047718539] 
\draw  [draw opacity=0] (159.2,89.49) .. controls (165,87.45) and (171.25,86.33) .. (177.76,86.33) .. controls (208.15,86.33) and (232.79,110.58) .. (232.79,140.5) .. controls (232.79,170.42) and (208.15,194.67) .. (177.76,194.67) .. controls (171.25,194.67) and (165,193.55) .. (159.2,191.51) -- (177.76,140.5) -- cycle ; \draw   (159.2,89.49) .. controls (165,87.45) and (171.25,86.33) .. (177.76,86.33) .. controls (208.15,86.33) and (232.79,110.58) .. (232.79,140.5) .. controls (232.79,170.42) and (208.15,194.67) .. (177.76,194.67) .. controls (171.25,194.67) and (165,193.55) .. (159.2,191.51) ;  
%Curve Lines [id:da38749161698780266] 
\draw    (105.81,191.77) .. controls (133.43,179.35) and (152.66,188.32) .. (159.97,191.77) ;
%Curve Lines [id:da3822763557092108] 
\draw    (105.81,89.23) .. controls (125.95,96.42) and (137.7,100.9) .. (159.97,89.23) ;
%Shape: Chord [id:dp9572449884735401] 
\draw   (100.72,175.4) .. controls (96.76,176.87) and (92.49,177.67) .. (88.02,177.67) .. controls (67.66,177.67) and (51.16,161.03) .. (51.16,140.5) .. controls (51.16,119.97) and (67.66,103.33) .. (88.02,103.33) .. controls (92.49,103.33) and (96.76,104.13) .. (100.72,105.6) -- cycle ;
%Shape: Chord [id:dp7031699571823893] 
\draw   (165.06,105.6) .. controls (169.02,104.13) and (173.3,103.33) .. (177.76,103.33) .. controls (198.12,103.33) and (214.62,119.97) .. (214.62,140.5) .. controls (214.62,161.03) and (198.12,177.67) .. (177.76,177.67) .. controls (173.3,177.67) and (169.02,176.87) .. (165.06,175.4) -- cycle ;
%Shape: Arc [id:dp7381069650732639] 
\draw  [draw opacity=0] (132.76,176.68) .. controls (121.67,173.06) and (113.4,158.94) .. (113.4,142.07) .. controls (113.4,124.67) and (122.19,110.2) .. (133.8,107.15) -- (138.37,142.07) -- cycle ; \draw   (132.76,176.68) .. controls (121.67,173.06) and (113.4,158.94) .. (113.4,142.07) .. controls (113.4,124.67) and (122.19,110.2) .. (133.8,107.15) ;  
%Shape: Arc [id:dp9657350777634695] 
\draw  [draw opacity=0] (133.13,107.28) .. controls (144.31,110.8) and (152.68,124.99) .. (152.68,141.95) .. controls (152.68,158.93) and (144.29,173.13) .. (133.08,176.63) -- (127.7,141.95) -- cycle ; \draw   (133.13,107.28) .. controls (144.31,110.8) and (152.68,124.99) .. (152.68,141.95) .. controls (152.68,158.93) and (144.29,173.13) .. (133.08,176.63) ;  
%Straight Lines [id:da5857111058430062] 
\draw [line width=1.5]  [dash pattern={on 1.69pt off 2.76pt}]  (88.02,140.5) -- (126.33,140.42) ;
%Straight Lines [id:da1455411052036164] 
\draw [line width=1.5]  [dash pattern={on 1.69pt off 2.76pt}]  (142.83,140.17) -- (179.33,139.92) ;
%Straight Lines [id:da42391465428037267] 
\draw [line width=1.5]  [dash pattern={on 1.69pt off 2.76pt}]  (206.25,140.17) -- (242.75,140.17) ;
%Straight Lines [id:da8929495276951249] 
\draw    (267,141.33) -- (321,141.33) ;
\draw [shift={(323,141.33)}, rotate = 180] [color={rgb, 255:red, 0; green, 0; blue, 0 }  ][line width=0.75]    (10.93,-3.29) .. controls (6.95,-1.4) and (3.31,-0.3) .. (0,0) .. controls (3.31,0.3) and (6.95,1.4) .. (10.93,3.29)   ;

\end{tikzpicture}
}
\subfigure{

\tikzset{every picture/.style={line width=0.75pt}} %set default line width to 0.75pt        

\begin{tikzpicture}[x=0.75pt,y=0.75pt,yscale=-.7,xscale=.7]
%uncomment if require: \path (0,300); %set diagram left start at 0, and has height of 300

%Straight Lines [id:da3312890547440621] 
\draw    (270,150.33) -- (324,150.33) ;
\draw [shift={(326,150.33)}, rotate = 180] [color={rgb, 255:red, 0; green, 0; blue, 0 }  ][line width=0.75]    (10.93,-3.29) .. controls (6.95,-1.4) and (3.31,-0.3) .. (0,0) .. controls (3.31,0.3) and (6.95,1.4) .. (10.93,3.29)   ;
%Shape: Ellipse [id:dp8384252970980643] 
\draw   (355.35,147.22) .. controls (355.35,102.07) and (392.62,65.47) .. (438.6,65.47) .. controls (484.57,65.47) and (521.84,102.07) .. (521.84,147.22) .. controls (521.84,192.37) and (484.57,228.97) .. (438.6,228.97) .. controls (392.62,228.97) and (355.35,192.37) .. (355.35,147.22) -- cycle ;
%Shape: Free Drawing [id:dp8521735713299265] 
\draw  [color={rgb, 255:red, 255; green, 255; blue, 255 }  ,draw opacity=1 ][line width=3] [line join = round][line cap = round] (428.73,150.72) .. controls (429.04,150.72) and (429.35,150.72) .. (429.66,150.72) ;
%Shape: Free Drawing [id:dp4765276263764715] 
\draw  [color={rgb, 255:red, 255; green, 255; blue, 255 }  ,draw opacity=1 ][line width=3] [line join = round][line cap = round] (433.07,66.28) .. controls (438.22,65.84) and (443.94,63.4) .. (448.57,65.67) ;
%Shape: Free Drawing [id:dp3861577217736605] 
\draw  [color={rgb, 255:red, 255; green, 255; blue, 255 }  ,draw opacity=1 ][line width=3] [line join = round][line cap = round] (433.07,228.26) .. controls (436.27,228.67) and (439.45,229.48) .. (442.68,229.48) .. controls (445.52,229.48) and (453.74,228.84) .. (451.05,227.96) ;
%Shape: Free Drawing [id:dp5901332554430636] 
\draw  [color={rgb, 255:red, 255; green, 255; blue, 255 }  ,draw opacity=1 ][line width=3] [line join = round][line cap = round] (365.48,186.25) .. controls (368.67,191.16) and (369.47,198.56) .. (374.78,201.16) ;
%Shape: Free Drawing [id:dp6190772558081827] 
\draw  [color={rgb, 255:red, 255; green, 255; blue, 255 }  ,draw opacity=1 ][line width=3] [line join = round][line cap = round] (372.3,195.38) .. controls (373.02,196.6) and (373.75,197.82) .. (374.47,199.03) ;
%Shape: Free Drawing [id:dp06887532434658827] 
\draw  [color={rgb, 255:red, 255; green, 255; blue, 255 }  ,draw opacity=1 ][line width=3] [line join = round][line cap = round] (505.31,97.64) .. controls (508.51,103.02) and (511.72,108.4) .. (514.92,113.78) ;
%Shape: Free Drawing [id:dp04897135650966811] 
\draw  [color={rgb, 255:red, 255; green, 255; blue, 255 }  ,draw opacity=1 ][line width=3] [line join = round][line cap = round] (367.96,102.82) .. controls (367.96,108.72) and (360.52,111.38) .. (360.52,117.74) .. controls (360.52,120.46) and (362.85,112.79) .. (364.24,110.43) .. controls (364.69,109.66) and (368.85,104.21) .. (369.2,104.04) ;
%Shape: Free Drawing [id:dp27451519455301154] 
\draw  [color={rgb, 255:red, 255; green, 255; blue, 255 }  ,draw opacity=1 ][line width=3] [line join = round][line cap = round] (516.16,179.95) .. controls (516.16,174.07) and (519.09,168.54) .. (520.81,162.9) ;
%Shape: Arc [id:dp4316207341350131] 
\draw  [draw opacity=0][dash pattern={on 0.84pt off 2.51pt}] (365.04,184.26) .. controls (369.05,184.32) and (372.97,186.16) .. (375.48,189.59) .. controls (378.13,193.2) and (378.59,197.7) .. (377.12,201.6) -- (364.76,197.18) -- cycle ; \draw  [dash pattern={on 0.84pt off 2.51pt}] (365.04,184.26) .. controls (369.05,184.32) and (372.97,186.16) .. (375.48,189.59) .. controls (378.13,193.2) and (378.59,197.7) .. (377.12,201.6) ;  
%Shape: Arc [id:dp7187790350178733] 
\draw  [draw opacity=0][dash pattern={on 0.84pt off 2.51pt}] (369.47,101.52) .. controls (371.01,104.93) and (371.08,109.01) .. (369.31,112.71) .. controls (367.54,116.42) and (364.29,118.97) .. (360.63,119.98) -- (357.39,107.21) -- cycle ; \draw  [dash pattern={on 0.84pt off 2.51pt}] (369.47,101.52) .. controls (371.01,104.93) and (371.08,109.01) .. (369.31,112.71) .. controls (367.54,116.42) and (364.29,118.97) .. (360.63,119.98) ;  
%Shape: Arc [id:dp13450872342235387] 
\draw  [draw opacity=0][dash pattern={on 0.84pt off 2.51pt}] (450.64,66.1) .. controls (448.52,69.26) and (445.04,71.33) .. (441.09,71.37) .. controls (436.94,71.4) and (433.27,69.18) .. (431.13,65.79) -- (440.97,58.92) -- cycle ; \draw  [dash pattern={on 0.84pt off 2.51pt}] (450.64,66.1) .. controls (448.52,69.26) and (445.04,71.33) .. (441.09,71.37) .. controls (436.94,71.4) and (433.27,69.18) .. (431.13,65.79) ;  
%Shape: Arc [id:dp165713153784385] 
\draw  [draw opacity=0][dash pattern={on 0.84pt off 2.51pt}] (513.84,115.63) .. controls (510.27,114.37) and (507,111.58) .. (505,107.65) .. controls (502.89,103.49) and (502.67,98.97) .. (504.05,95.28) -- (516.36,102.09) -- cycle ; \draw  [dash pattern={on 0.84pt off 2.51pt}] (513.84,115.63) .. controls (510.27,114.37) and (507,111.58) .. (505,107.65) .. controls (502.89,103.49) and (502.67,98.97) .. (504.05,95.28) ;  
%Shape: Arc [id:dp48355274823754524] 
\draw  [draw opacity=0][dash pattern={on 0.84pt off 2.51pt}] (514.14,179.17) .. controls (511.94,176.22) and (511.16,172.36) .. (512.4,168.7) .. controls (513.71,164.85) and (516.93,162.13) .. (520.73,161.19) -- (523.94,172.47) -- cycle ; \draw  [dash pattern={on 0.84pt off 2.51pt}] (514.14,179.17) .. controls (511.94,176.22) and (511.16,172.36) .. (512.4,168.7) .. controls (513.71,164.85) and (516.93,162.13) .. (520.73,161.19) ;  
%Shape: Arc [id:dp05595748124531452] 
\draw  [draw opacity=0][dash pattern={on 0.84pt off 2.51pt}] (431.07,227.88) .. controls (433.18,224.21) and (437.17,221.64) .. (441.85,221.43) .. controls (447.14,221.19) and (451.83,224.04) .. (454.05,228.31) -- (442.43,233.86) -- cycle ; \draw  [dash pattern={on 0.84pt off 2.51pt}] (431.07,227.88) .. controls (433.18,224.21) and (437.17,221.64) .. (441.85,221.43) .. controls (447.14,221.19) and (451.83,224.04) .. (454.05,228.31) ;  

%Shape: Ellipse [id:dp5310293170856262] 
\draw   (72.35,148.22) .. controls (72.35,103.07) and (109.62,66.47) .. (155.6,66.47) .. controls (201.57,66.47) and (238.84,103.07) .. (238.84,148.22) .. controls (238.84,193.37) and (201.57,229.97) .. (155.6,229.97) .. controls (109.62,229.97) and (72.35,193.37) .. (72.35,148.22) -- cycle ;
%Shape: Free Drawing [id:dp951921792490811] 
\draw  [color={rgb, 255:red, 255; green, 255; blue, 255 }  ,draw opacity=1 ][line width=3] [line join = round][line cap = round] (145.73,151.72) .. controls (146.04,151.72) and (146.35,151.72) .. (146.66,151.72) ;
%Shape: Free Drawing [id:dp6397930195270172] 
\draw  [color={rgb, 255:red, 255; green, 255; blue, 255 }  ,draw opacity=1 ][line width=3] [line join = round][line cap = round] (150.07,67.28) .. controls (155.22,66.84) and (160.94,64.4) .. (165.57,66.67) ;
%Shape: Free Drawing [id:dp6838947531865482] 
\draw  [color={rgb, 255:red, 255; green, 255; blue, 255 }  ,draw opacity=1 ][line width=3] [line join = round][line cap = round] (150.07,229.26) .. controls (153.27,229.67) and (156.45,230.48) .. (159.68,230.48) .. controls (162.52,230.48) and (170.74,229.84) .. (168.05,228.96) ;
%Shape: Free Drawing [id:dp2837347111054711] 
\draw  [color={rgb, 255:red, 255; green, 255; blue, 255 }  ,draw opacity=1 ][line width=3] [line join = round][line cap = round] (82.48,187.25) .. controls (85.67,192.16) and (86.47,199.56) .. (91.78,202.16) ;
%Shape: Free Drawing [id:dp9359123744778277] 
\draw  [color={rgb, 255:red, 255; green, 255; blue, 255 }  ,draw opacity=1 ][line width=3] [line join = round][line cap = round] (89.3,196.38) .. controls (90.02,197.6) and (90.75,198.82) .. (91.47,200.03) ;
%Shape: Free Drawing [id:dp3729931080179709] 
\draw  [color={rgb, 255:red, 255; green, 255; blue, 255 }  ,draw opacity=1 ][line width=3] [line join = round][line cap = round] (222.31,98.64) .. controls (225.51,104.02) and (228.72,109.4) .. (231.92,114.78) ;
%Shape: Free Drawing [id:dp5040756696267854] 
\draw  [color={rgb, 255:red, 255; green, 255; blue, 255 }  ,draw opacity=1 ][line width=3] [line join = round][line cap = round] (84.96,103.82) .. controls (84.96,109.72) and (77.52,112.38) .. (77.52,118.74) .. controls (77.52,121.46) and (79.85,113.79) .. (81.24,111.43) .. controls (81.69,110.66) and (85.85,105.21) .. (86.2,105.04) ;
%Shape: Free Drawing [id:dp7296988241042062] 
\draw  [color={rgb, 255:red, 255; green, 255; blue, 255 }  ,draw opacity=1 ][line width=3] [line join = round][line cap = round] (233.16,180.95) .. controls (233.16,175.07) and (236.09,169.54) .. (237.81,163.9) ;
%Shape: Arc [id:dp11227299621419573] 
\draw  [draw opacity=0][dash pattern={on 0.84pt off 2.51pt}] (82.04,185.26) .. controls (86.05,185.32) and (89.97,187.16) .. (92.48,190.59) .. controls (95.13,194.2) and (95.59,198.7) .. (94.12,202.6) -- (81.76,198.18) -- cycle ; \draw  [dash pattern={on 0.84pt off 2.51pt}] (82.04,185.26) .. controls (86.05,185.32) and (89.97,187.16) .. (92.48,190.59) .. controls (95.13,194.2) and (95.59,198.7) .. (94.12,202.6) ;  
%Shape: Arc [id:dp10991677377697884] 
\draw  [draw opacity=0][dash pattern={on 0.84pt off 2.51pt}] (86.47,102.52) .. controls (88.01,105.93) and (88.08,110.01) .. (86.31,113.71) .. controls (84.54,117.42) and (81.29,119.97) .. (77.63,120.98) -- (74.39,108.21) -- cycle ; \draw  [dash pattern={on 0.84pt off 2.51pt}] (86.47,102.52) .. controls (88.01,105.93) and (88.08,110.01) .. (86.31,113.71) .. controls (84.54,117.42) and (81.29,119.97) .. (77.63,120.98) ;  
%Shape: Arc [id:dp06956220479464514] 
\draw  [draw opacity=0][dash pattern={on 0.84pt off 2.51pt}] (167.64,67.1) .. controls (165.52,70.26) and (162.04,72.33) .. (158.09,72.37) .. controls (153.94,72.4) and (150.27,70.18) .. (148.13,66.79) -- (157.97,59.92) -- cycle ; \draw  [dash pattern={on 0.84pt off 2.51pt}] (167.64,67.1) .. controls (165.52,70.26) and (162.04,72.33) .. (158.09,72.37) .. controls (153.94,72.4) and (150.27,70.18) .. (148.13,66.79) ;  
%Shape: Arc [id:dp7823666062404169] 
\draw  [draw opacity=0][dash pattern={on 0.84pt off 2.51pt}] (230.84,116.63) .. controls (227.27,115.37) and (224,112.58) .. (222,108.65) .. controls (219.89,104.49) and (219.67,99.97) .. (221.05,96.28) -- (233.36,103.09) -- cycle ; \draw  [dash pattern={on 0.84pt off 2.51pt}] (230.84,116.63) .. controls (227.27,115.37) and (224,112.58) .. (222,108.65) .. controls (219.89,104.49) and (219.67,99.97) .. (221.05,96.28) ;  
%Shape: Arc [id:dp0807438852138258] 
\draw  [draw opacity=0][dash pattern={on 0.84pt off 2.51pt}] (231.14,180.17) .. controls (228.94,177.22) and (228.16,173.36) .. (229.4,169.7) .. controls (230.71,165.85) and (233.93,163.13) .. (237.73,162.19) -- (240.94,173.47) -- cycle ; \draw  [dash pattern={on 0.84pt off 2.51pt}] (231.14,180.17) .. controls (228.94,177.22) and (228.16,173.36) .. (229.4,169.7) .. controls (230.71,165.85) and (233.93,163.13) .. (237.73,162.19) ;  
%Shape: Arc [id:dp15804707472699087] 
\draw  [draw opacity=0][dash pattern={on 0.84pt off 2.51pt}] (148.07,228.88) .. controls (150.18,225.21) and (154.17,222.64) .. (158.85,222.43) .. controls (164.14,222.19) and (168.83,225.04) .. (171.05,229.31) -- (159.43,234.86) -- cycle ; \draw  [dash pattern={on 0.84pt off 2.51pt}] (148.07,228.88) .. controls (150.18,225.21) and (154.17,222.64) .. (158.85,222.43) .. controls (164.14,222.19) and (168.83,225.04) .. (171.05,229.31) ;  
%Shape: Arc [id:dp325127530248851] 
\draw  [draw opacity=0][dash pattern={on 0.84pt off 2.51pt}] (75.18,95.63) .. controls (79.91,100.92) and (81.45,107.28) .. (78.51,111.64) .. controls (75.78,115.67) and (69.91,116.9) .. (63.65,115.3) -- (62.76,101.41) -- cycle ; \draw  [dash pattern={on 0.84pt off 2.51pt}] (75.18,95.63) .. controls (79.91,100.92) and (81.45,107.28) .. (78.51,111.64) .. controls (75.78,115.67) and (69.91,116.9) .. (63.65,115.3) ;  
%Straight Lines [id:da8408680252459675] 
\draw    (75.34,95.82) -- (63.7,115.31) ;

%Shape: Arc [id:dp28035919261754816] 
\draw  [draw opacity=0][dash pattern={on 0.84pt off 2.51pt}] (170.63,55.91) .. controls (167.96,62.36) and (162.91,66.54) .. (157.67,65.94) .. controls (152.81,65.38) and (149.06,60.89) .. (147.64,54.76) -- (159.81,47.52) -- cycle ; \draw  [dash pattern={on 0.84pt off 2.51pt}] (170.63,55.91) .. controls (167.96,62.36) and (162.91,66.54) .. (157.67,65.94) .. controls (152.81,65.38) and (149.06,60.89) .. (147.64,54.76) ;  
%Straight Lines [id:da1667159350344889] 
\draw    (170.59,55.99) -- (147.62,54.64) ;

%Shape: Arc [id:dp266810052969473] 
\draw  [draw opacity=0][dash pattern={on 0.84pt off 2.51pt}] (244.01,110.31) .. controls (237.01,111.87) and (230.57,110.16) .. (228.08,105.53) .. controls (225.78,101.24) and (227.49,95.59) .. (231.94,90.95) -- (244.76,96.85) -- cycle ; \draw  [dash pattern={on 0.84pt off 2.51pt}] (244.01,110.31) .. controls (237.01,111.87) and (230.57,110.16) .. (228.08,105.53) .. controls (225.78,101.24) and (227.49,95.59) .. (231.94,90.95) ;  
%Straight Lines [id:da939440508862887] 
\draw    (243.75,110.37) -- (231.88,91) ;

%Shape: Arc [id:dp22089949404596165] 
\draw  [draw opacity=0][dash pattern={on 0.84pt off 2.51pt}] (245.61,186.77) .. controls (239.33,183.33) and (235.68,177.82) .. (236.99,172.74) .. controls (238.2,168.03) and (243.34,164.95) .. (249.82,164.42) -- (255.3,177.23) -- cycle ; \draw  [dash pattern={on 0.84pt off 2.51pt}] (245.61,186.77) .. controls (239.33,183.33) and (235.68,177.82) .. (236.99,172.74) .. controls (238.2,168.03) and (243.34,164.95) .. (249.82,164.42) ;  
%Straight Lines [id:da6539223288574709] 
\draw    (245.34,186.62) -- (249.7,164.43) ;

%Shape: Arc [id:dp42755221987745773] 
\draw  [draw opacity=0][dash pattern={on 0.84pt off 2.51pt}] (69.79,192.42) .. controls (76.45,189.84) and (83.07,190.55) .. (86.25,194.74) .. controls (89.19,198.62) and (88.39,204.45) .. (84.75,209.7) -- (71.13,205.85) -- cycle ; \draw  [dash pattern={on 0.84pt off 2.51pt}] (69.79,192.42) .. controls (76.45,189.84) and (83.07,190.55) .. (86.25,194.74) .. controls (89.19,198.62) and (88.39,204.45) .. (84.75,209.7) ;  
%Straight Lines [id:da9328921667462142] 
\draw    (70.01,192.33) -- (84.77,209.67) ;

%Shape: Arc [id:dp20805739172296134] 
\draw  [draw opacity=0][dash pattern={on 0.84pt off 2.51pt}] (147.82,240.82) .. controls (150.04,234.2) and (154.79,229.71) .. (160.06,229.96) .. controls (164.94,230.19) and (168.99,234.43) .. (170.83,240.44) -- (159.19,248.47) -- cycle ; \draw  [dash pattern={on 0.84pt off 2.51pt}] (147.82,240.82) .. controls (150.04,234.2) and (154.79,229.71) .. (160.06,229.96) .. controls (164.94,230.19) and (168.99,234.43) .. (170.83,240.44) ;  
%Straight Lines [id:da013887530102412393] 
\draw    (147.85,240.73) -- (170.86,240.57) ;

% Text Node
\draw (430.05,40.01) node [anchor=north west][inner sep=0.75pt]  [font=\Huge,rotate=-271.67,xslant=-0.03]  {$\mathcal{\simeq }$};

\end{tikzpicture}

}
    \caption{Transforming $K$ into a disk with cut out lunettes in the case when $\pi(\overline{\Sigma}_{C,h})$ is a disk with three holes. }
    \label{fig: we cut out the BALLS}
\end{figure}
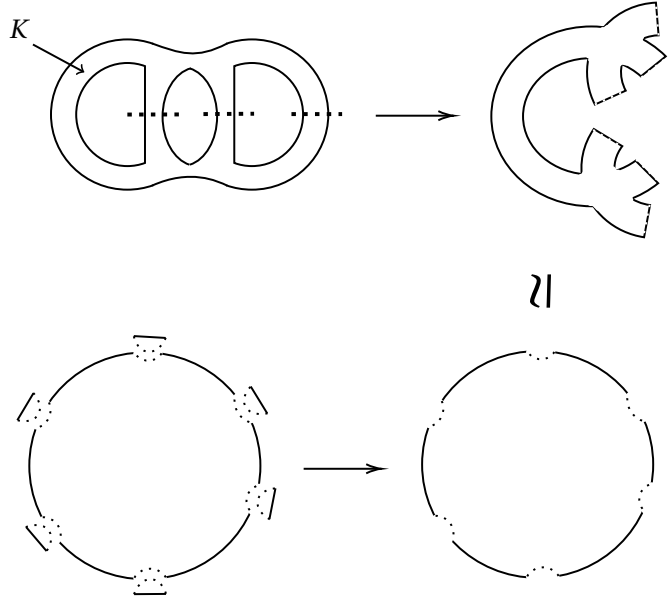
\end{proof}

For a discussion of the bifurcation diagram pertaining to the appearance of these holes in the image of $\pi$ see Subsection \ref{bif.sub1}. There we will also show that this bifurcation diagram coincides with the bifurcation diagram describing the existence of relative equilibria in \cite{borisov2018reduction}.

\begin{theorem}
$\overline{\Sigma}_{C,h}$ is homeomorphic to one of the following manifolds:
\begin{enumerate}
    \item $S^1\times S^2$, if $\pi:\overline{\Sigma}_{C,h}\rightarrow S^2_{C}$ is surjective or its image is the sphere minus two open disks.
    \item $\left(S^1\times S^2\right)\#\left(S^1\times S^2\right)\#\left(S^1\times S^2\right)$, if the image of $\pi:\overline{\Sigma}_{C,h}\rightarrow S^2_{C}$  is the sphere minus four open disks, where $M\#N$ denotes the connected sum of manifolds $M$ and $N$.
\end{enumerate}
\end{theorem}
\begin{proof}
%The techniques to prove this kind of results are thoroughly developed in \cite{bolsinov2004integrable}. Here we give a direct proof leveraging on the fact that we are dealing with a concrete situation. 
{\bf Case 1} The simplest case is when $\pi$ is surjective. To be clear, we denote with $S^2_C$ the sphere corresponding to the Casimir and with $S^2$ the sphere corresponding to the compactication of $\mathbb{R}^2.$ By Proposition \eqref{prop3.2} we already know that $\overline{\Sigma}_{C,h}$ is a hypersurface in $S^2\times S^2_C.$ For each $p\in S^2_C$, we have a corresponding simple continuous closed curve in $S^2$ given explicitly by an equation. Let $\phi_p:S^1 \rightarrow S^2$ be a continuous paramaterisation of the simple continuous curve in $S^2.$ This depends continuously on $p\in S^2_C$ and therefore $\Psi: S^1\times S^2_C\rightarrow S^2\times S^2_C$ given by $\Psi(t,p)=(\phi_p(t),p)$ is continuous. Furthermore, notice that $\phi_p$ is a homeomorphism onto its image for each $p\in S^2_C$ and that $pr_2\circ\Psi=pr_2$. Therefore, $\Psi$ is a homeomorphism onto its image $W:=\Psi(S^1\times S^2_C)=\overline{\Sigma}_{C,h}\subset S^2\times S^2_C.$ Now $\pi: \overline{\Sigma}_{C,h}\rightarrow S^2_C$ is the same as $pr_2{|_W}.$ Thus $\Psi^{-1}:W\rightarrow S^1\times S^2_C$ is a global continuous trivialisation with $pr_2\circ \Psi=pr_2{|_W}.$

%Since there is no additional underlying geometric constraint (i.e. relations between $p, \ \xi$ and $m_i$s), the fibration is trivial and $\overline{\Sigma}_{C,h}$ is homeomophic to $ S^1\times S^2$. 
    
{\bf Case 2} In the second case, the image of $\pi$ is $S^2\setminus\{\Delta_1\cup \Delta_2\}$ where $\Delta_1$ and $\Delta_2$ are open disks. Therefore the image of $\pi$ is diffeomorphic to the compact cylinder $S^1\times [0,1]$. If we remove also the boundaries of these disks and apply Case 1, we see that what we obtain is homeomorphic to the trivial fibration with base
the open cylinder $S^1\times (0,1)$ and fibre $S^1.$
Over each boundary point of this cylinder, the inverse image of $\pi$ is given by a point while over each point in $S^1\times (0,1)$ the inverse image of $\pi$ is a circle. Now if we look at the inverse image under $\pi$ of any ruling ${Q}\times [0,1]$, where $Q$ is a fixed point in $S^1$ we get that $\pi^{-1}({Q}\times [0,1])$ is homeomorphic to a sphere as it is immediate to see (c.f. Figure \ref{fig:preimage pi}). As $Q$ varies in $S^1$, it is immediate to see that in this case $\overline{\Sigma}_{C,h}$ is homeomorphic to $S^1\times S^2.$

{\bf Case 3} The third case is the most involved. Before dealing with it, we introduce some preliminary constructions. First observe that we can represent a 3-dimensional sphere $S^3$ as a fibration $\Pi: S^3 \rightarrow \overline{\Delta}$ over a closed disk: the fibres over the interior points are homeomorphic to $S^1$ and those over the boundary points are just points. This is given by $$\Pi: \{x\in \mathbb{R}^4: x_1^2+x_2^2+x_3^2+x_4^2=1\} \rightarrow \{(x_1,x_2)\in \mathbb{R}^2: x_1^2+x_2^2\leq 1\}.$$ Now consider a closed lunette $L\subset \overline{\Delta}$ as the ones we cut from a two-dimensional disk in the  bottom left part of  Figure \ref{fig: we cut out the BALLS} and the inverse image $\Pi^{-1}(L).$ It is clear that this is just homeomorphic to a 3-dimensional ball $B^3$ since the inverse image of $\Pi$ above every point in $L\setminus \partial\overline{\Delta}$ is a circle while above every point in $L\cap\partial\overline{\Delta}$ the inverse image is a point. So $\Pi^{-1}(L)$ is like having the lunette rotate by $360$ degrees around the segment $L\cap \partial\overline{\Delta}.$ 

For analysing the third case, we also need the following preliminary construction. With reference to the previous paragraph, consider again $\Pi: S^3 \rightarrow \overline{\Delta}.$ This time we consider two disjoint closed lunettes $L_1,L_2$ in $\overline{\Delta}.$ Consider $\Pi^{-1}(\overline{\Delta}\setminus{L_1\cup L_2}),$ this is of course just $S^3$ with two disjoint closed balls $B^3_1,B^3_2$ removed. We claim that if we glue $S^3$ identifying two the boundaries $\partial B^3_1, \partial B^3_2$ the resulting manifold is homeomorphic to $S^1\times S^2.$ To see this, represent $S^3$ in $\mathbb{R}^4$ as $S^3=\{x\in \mathbb{R}^4: x_1^2+x_2^2+x_3^2+x_4^2=1\}$ and the two closed balls as $B_1^3:=\{x\in S^3: x_1^2+x_2^2+x_3^2\leq \epsilon^2, \; x_4>0 \}$ and $B_1^3:=\{x\in S^3: x_1^2+x_2^2+x_3^2\leq \epsilon^2, \; x_4<0 \}.$ Using the meridians corresponding to the coordinate $x_4$, it is clear that any point $Q$ in $S^3\setminus\{B_1^3\cup B_2^3\}$ after the gluing of the 2-dimensional spheres $\partial B_1^3$ and $\partial B_2^3$ can be represented by a point $P(Q)$ on $S^2$ (obtained by intersecting the meridian going through $Q$ with $S^2$) and a coordinate on the meridian. But under identification of $\partial B_1^3$ and $\partial B_2^3$, each meridian becomes a circle (each meridian intersects the two boundaries at two different points which are then identified under gluing). Therefore any point $P$ on $S^3\setminus\{B_1^3\cup B_2^3\}/~$ ($/~$ represents the operation of gluing) is uniquely identified with $(Q(P),\theta(P))\in S^2\times S^1$. Thus $S^3\setminus\{B_1^3\cup B_2^3\}/~$ is homeomorphic to $S^2\times S^1$. For illustration, see Figure \ref{fig:S3}.

\begin{figure}
    \centering

% Pattern Info
 
\tikzset{
pattern size/.store in=\mcSize, 
pattern size = 5pt,
pattern thickness/.store in=\mcThickness, 
pattern thickness = 0.3pt,
pattern radius/.store in=\mcRadius, 
pattern radius = 1pt}
\makeatletter
\pgfutil@ifundefined{pgf@pattern@name@_yum47bshk}{
\pgfdeclarepatternformonly[\mcThickness,\mcSize]{_yum47bshk}
{\pgfqpoint{0pt}{0pt}}
{\pgfpoint{\mcSize+\mcThickness}{\mcSize+\mcThickness}}
{\pgfpoint{\mcSize}{\mcSize}}
{
\pgfsetcolor{\tikz@pattern@color}
\pgfsetlinewidth{\mcThickness}
\pgfpathmoveto{\pgfqpoint{0pt}{0pt}}
\pgfpathlineto{\pgfpoint{\mcSize+\mcThickness}{\mcSize+\mcThickness}}
\pgfusepath{stroke}
}}
\makeatother

% Pattern Info
 
\tikzset{
pattern size/.store in=\mcSize, 
pattern size = 5pt,
pattern thickness/.store in=\mcThickness, 
pattern thickness = 0.3pt,
pattern radius/.store in=\mcRadius, 
pattern radius = 1pt}
\makeatletter
\pgfutil@ifundefined{pgf@pattern@name@_p2oy0bcei}{
\pgfdeclarepatternformonly[\mcThickness,\mcSize]{_p2oy0bcei}
{\pgfqpoint{0pt}{-\mcThickness}}
{\pgfpoint{\mcSize}{\mcSize}}
{\pgfpoint{\mcSize}{\mcSize}}
{
\pgfsetcolor{\tikz@pattern@color}
\pgfsetlinewidth{\mcThickness}
\pgfpathmoveto{\pgfqpoint{0pt}{\mcSize}}
\pgfpathlineto{\pgfpoint{\mcSize+\mcThickness}{-\mcThickness}}
\pgfusepath{stroke}
}}
\makeatother
\tikzset{every picture/.style={line width=0.75pt}} %set default line width to 0.75pt        

\begin{tikzpicture}[x=0.75pt,y=0.75pt,yscale=-1,xscale=1]
%uncomment if require: \path (0,327); %set diagram left start at 0, and has height of 327

%Shape: Arc [id:dp4973211311710868] 
\draw  [draw opacity=0][pattern=_yum47bshk,pattern size=6pt,pattern thickness=0.75pt,pattern radius=0pt, pattern color={rgb, 255:red, 189; green, 16; blue, 224}] (135.73,232.56) .. controls (143.47,226.6) and (160.37,222.61) .. (179.85,222.88) .. controls (198.41,223.14) and (214.4,227.19) .. (222.28,232.92) -- (179.36,240.44) -- cycle ; \draw  [color={rgb, 255:red, 245; green, 166; blue, 35 }  ,draw opacity=1 ] (135.73,232.56) .. controls (143.47,226.6) and (160.37,222.61) .. (179.85,222.88) .. controls (198.41,223.14) and (214.4,227.19) .. (222.28,232.92) ;  
%Shape: Arc [id:dp6063459565267404] 
\draw  [draw opacity=0][pattern=_p2oy0bcei,pattern size=6pt,pattern thickness=0.75pt,pattern radius=0pt, pattern color={rgb, 255:red, 189; green, 16; blue, 224}] (216.12,68.25) .. controls (207.24,72.1) and (193.48,74.54) .. (178.09,74.46) .. controls (163.6,74.39) and (150.65,72.1) .. (141.89,68.55) -- (178.4,57.78) -- cycle ; \draw  [color={rgb, 255:red, 245; green, 166; blue, 35 }  ,draw opacity=1 ] (216.12,68.25) .. controls (207.24,72.1) and (193.48,74.54) .. (178.09,74.46) .. controls (163.6,74.39) and (150.65,72.1) .. (141.89,68.55) ;  
%Shape: Ellipse [id:dp8339781722921529] 
\draw  [dash pattern={on 4.5pt off 4.5pt}] (89.78,150) .. controls (89.78,142.88) and (129.82,137.11) .. (179.22,137.11) .. controls (228.62,137.11) and (268.67,142.88) .. (268.67,150) .. controls (268.67,157.12) and (228.62,162.89) .. (179.22,162.89) .. controls (129.82,162.89) and (89.78,157.12) .. (89.78,150) -- cycle ;
%Shape: Free Drawing [id:dp1688370415617726] 
\draw  [color={rgb, 255:red, 255; green, 255; blue, 255 }  ,draw opacity=1 ][line width=0.75] [line join = round][line cap = round] (136,236.5) .. controls (135.75,234.92) and (134.38,233.1) .. (135.25,231.75) .. controls (135.79,230.9) and (141.6,235.51) .. (136.25,231.5) ;
%Shape: Free Drawing [id:dp9139821922446505] 
\draw  [color={rgb, 255:red, 255; green, 255; blue, 255 }  ,draw opacity=1 ][line width=0.75] [line join = round][line cap = round] (136.5,231.5) .. controls (137.58,232.42) and (140.54,233.07) .. (139.75,234.25) .. controls (138.96,235.43) and (137.34,232.7) .. (136,232.25) .. controls (135.14,231.96) and (137.69,234.45) .. (138.25,233.75) .. controls (138.88,232.96) and (137.42,231.92) .. (137,231) ;
%Shape: Free Drawing [id:dp7709754161133267] 
\draw  [color={rgb, 255:red, 255; green, 255; blue, 255 }  ,draw opacity=1 ][line width=0.75] [line join = round][line cap = round] (141.25,234.25) .. controls (142.17,235.67) and (144.53,236.9) .. (144,238.5) .. controls (143.58,239.75) and (141.55,237.46) .. (140.25,237.25) .. controls (138.93,237.04) and (135.06,237.85) .. (136.25,237.25) .. controls (138.53,236.11) and (141.75,235.59) .. (142.75,233.25) ;
%Shape: Free Drawing [id:dp6898082780578461] 
\draw  [color={rgb, 255:red, 255; green, 255; blue, 255 }  ,draw opacity=1 ][line width=0.75] [line join = round][line cap = round] (142.75,237.25) .. controls (141.58,236.58) and (140.12,236.27) .. (139.25,235.25) .. controls (138.91,234.85) and (139.48,233.3) .. (139.75,233.75) .. controls (140.46,234.93) and (141.08,236.42) .. (140.75,237.75) .. controls (140.55,238.56) and (139.75,236.42) .. (139.25,235.75) ;
%Shape: Free Drawing [id:dp2719750529049725] 
\draw  [color={rgb, 255:red, 255; green, 255; blue, 255 }  ,draw opacity=1 ][line width=0.75] [line join = round][line cap = round] (136.25,230) .. controls (137.25,230.58) and (138.25,231.17) .. (139.25,231.75) ;
%Shape: Free Drawing [id:dp29583602419966937] 
\draw  [color={rgb, 255:red, 255; green, 255; blue, 255 }  ,draw opacity=1 ][line width=0.75] [line join = round][line cap = round] (137.5,230.75) .. controls (138.08,230.83) and (138.67,230.92) .. (139.25,231) ;
%Shape: Free Drawing [id:dp13859931460776043] 
\draw  [color={rgb, 255:red, 255; green, 255; blue, 255 }  ,draw opacity=1 ][line width=0.75] [line join = round][line cap = round] (141.25,232.75) .. controls (145.32,234.63) and (145.49,234.91) .. (142.25,232.75) ;
%Shape: Free Drawing [id:dp9727175080402342] 
\draw  [color={rgb, 255:red, 255; green, 255; blue, 255 }  ,draw opacity=1 ][line width=0.75] [line join = round][line cap = round] (141.75,232.5) .. controls (142.58,233) and (143.42,233.5) .. (144.25,234) ;
%Shape: Free Drawing [id:dp6002068927885982] 
\draw  [color={rgb, 255:red, 255; green, 255; blue, 255 }  ,draw opacity=1 ][line width=0.75] [line join = round][line cap = round] (214.75,234.75) .. controls (213.25,235.17) and (208.95,236.86) .. (210.25,236) .. controls (212.27,234.65) and (214.68,233.97) .. (217,233.25) .. controls (218,232.94) and (215.24,234.42) .. (214.25,234.75) .. controls (213.5,235) and (215.58,233.92) .. (216.25,233.5) ;
%Shape: Free Drawing [id:dp2984958938063367] 
\draw  [color={rgb, 255:red, 255; green, 255; blue, 255 }  ,draw opacity=1 ][line width=0.75] [line join = round][line cap = round] (219,233.75) .. controls (220.83,232.33) and (222.46,230.6) .. (224.5,229.5) .. controls (229.44,226.84) and (219.76,234.51) .. (218.25,233.75) .. controls (215.91,232.58) and (222.4,229.46) .. (225,229.75) .. controls (226.77,229.95) and (216.29,233.56) .. (220.5,230.75) ;
%Shape: Free Drawing [id:dp9916202199357229] 
\draw  [color={rgb, 255:red, 255; green, 255; blue, 255 }  ,draw opacity=1 ][line width=0.75] [line join = round][line cap = round] (225.25,231.75) .. controls (223.75,232.75) and (222.5,234.31) .. (220.75,234.75) .. controls (218.95,235.2) and (222.24,229.66) .. (224,230.25) .. controls (228.5,231.75) and (218.43,235.05) .. (220,234) .. controls (221.77,232.82) and (223.83,232.17) .. (225.75,231.25) ;
%Shape: Free Drawing [id:dp47010867544778345] 
\draw  [color={rgb, 255:red, 255; green, 255; blue, 255 }  ,draw opacity=1 ][line width=0.75] [line join = round][line cap = round] (214.25,234.25) .. controls (214.9,234.09) and (215.44,233.62) .. (216,233.25) ;
%Shape: Free Drawing [id:dp23967185997568397] 
\draw  [color={rgb, 255:red, 255; green, 255; blue, 255 }  ,draw opacity=1 ][line width=0.75] [line join = round][line cap = round] (215.5,233.75) .. controls (215.92,233.42) and (216.33,233.08) .. (216.75,232.75) ;
%Shape: Free Drawing [id:dp06574519160891001] 
\draw  [color={rgb, 255:red, 255; green, 255; blue, 255 }  ,draw opacity=1 ][line width=0.75] [line join = round][line cap = round] (213.75,234.25) .. controls (214.5,234) and (215.25,233.75) .. (216,233.5) ;
%Shape: Free Drawing [id:dp5410795837203348] 
\draw  [color={rgb, 255:red, 255; green, 255; blue, 255 }  ,draw opacity=1 ][line width=0.75] [line join = round][line cap = round] (172.75,59.25) .. controls (174.44,58.83) and (174.07,58.88) .. (171.75,59.75) ;
%Shape: Free Drawing [id:dp37157217912516893] 
\draw  [color={rgb, 255:red, 255; green, 255; blue, 255 }  ,draw opacity=1 ][line width=0.75] [line join = round][line cap = round] (177.75,56.75) .. controls (177.75,57.08) and (177.75,57.42) .. (177.75,57.75) ;
%Shape: Free Drawing [id:dp8221396641761769] 
\draw  [color={rgb, 255:red, 255; green, 255; blue, 255 }  ,draw opacity=1 ][line width=0.75] [line join = round][line cap = round] (178.25,58.25) .. controls (178.25,58.33) and (178.25,58.42) .. (178.25,58.5) ;
%Shape: Free Drawing [id:dp7835549872106764] 
\draw  [color={rgb, 255:red, 255; green, 255; blue, 255 }  ,draw opacity=1 ][line width=0.75] [line join = round][line cap = round] (215.75,233.17) .. controls (213.7,234.88) and (213.72,234.79) .. (216,233.17) ;
%Shape: Free Drawing [id:dp9464376958339769] 
\draw  [color={rgb, 255:red, 255; green, 255; blue, 255 }  ,draw opacity=1 ][line width=0.75] [line join = round][line cap = round] (215.25,233.92) .. controls (214.25,234.92) and (214.25,234.92) .. (215.25,233.92) ;
%Shape: Free Drawing [id:dp4372584270487567] 
\draw  [color={rgb, 255:red, 255; green, 255; blue, 255 }  ,draw opacity=1 ][line width=0.75] [line join = round][line cap = round] (194.5,24.92) .. controls (194.5,25.08) and (194.5,25.25) .. (194.5,25.42) ;
%Straight Lines [id:da10483261653430942] 
\draw    (121.25,236.92) -- (135.57,232.95) ;
\draw [shift={(137.5,232.42)}, rotate = 164.52] [color={rgb, 255:red, 0; green, 0; blue, 0 }  ][line width=0.75]    (10.93,-3.29) .. controls (6.95,-1.4) and (3.31,-0.3) .. (0,0) .. controls (3.31,0.3) and (6.95,1.4) .. (10.93,3.29)   ;
%Straight Lines [id:da24563211634844762] 
\draw    (234,50.83) -- (218.82,62.84) ;
\draw [shift={(217.25,64.08)}, rotate = 321.65] [color={rgb, 255:red, 0; green, 0; blue, 0 }  ][line width=0.75]    (10.93,-3.29) .. controls (6.95,-1.4) and (3.31,-0.3) .. (0,0) .. controls (3.31,0.3) and (6.95,1.4) .. (10.93,3.29)   ;
%Shape: Circle [id:dp11623940206510475] 
\draw  [fill={rgb, 255:red, 0; green, 0; blue, 0 }  ,fill opacity=1 ] (201.14,72.67) .. controls (201.14,71.84) and (201.81,71.17) .. (202.64,71.17) .. controls (203.46,71.17) and (204.14,71.84) .. (204.14,72.67) .. controls (204.14,73.49) and (203.46,74.17) .. (202.64,74.17) .. controls (201.81,74.17) and (201.14,73.49) .. (201.14,72.67) -- cycle ;
%Shape: Circle [id:dp9515351622944976] 
\draw  [fill={rgb, 255:red, 0; green, 0; blue, 0 }  ,fill opacity=1 ] (200.48,225.05) .. controls (200.48,224.23) and (201.15,223.55) .. (201.98,223.55) .. controls (202.8,223.55) and (203.48,224.23) .. (203.48,225.05) .. controls (203.48,225.88) and (202.8,226.55) .. (201.98,226.55) .. controls (201.15,226.55) and (200.48,225.88) .. (200.48,225.05) -- cycle ;
%Shape: Circle [id:dp6630027549229276] 
\draw  [fill={rgb, 255:red, 0; green, 0; blue, 0 }  ,fill opacity=1 ] (228.84,119.97) .. controls (228.84,119.14) and (229.51,118.47) .. (230.34,118.47) .. controls (231.16,118.47) and (231.84,119.14) .. (231.84,119.97) .. controls (231.84,120.79) and (231.16,121.47) .. (230.34,121.47) .. controls (229.51,121.47) and (228.84,120.79) .. (228.84,119.97) -- cycle ;
%Shape: Arc [id:dp4965568383752976] 
\draw  [draw opacity=0] (179.13,60.5) .. controls (213.21,67.54) and (236.78,109.18) .. (232.56,157.02) .. controls (229.02,197.18) and (206.96,229.61) .. (179.39,239.56) -- (165.28,151.08) -- cycle ; \draw   (179.13,60.5) .. controls (213.21,67.54) and (236.78,109.18) .. (232.56,157.02) .. controls (229.02,197.18) and (206.96,229.61) .. (179.39,239.56) ;  
%Shape: Arc [id:dp8349825362769492] 
\draw  [draw opacity=0] (202.64,224.38) .. controls (192.04,209.35) and (184.88,180.62) .. (184.88,147.65) .. controls (184.88,115.85) and (191.54,87.99) .. (201.53,72.57) -- (219.52,147.65) -- cycle ; \draw   (202.64,224.38) .. controls (192.04,209.35) and (184.88,180.62) .. (184.88,147.65) .. controls (184.88,115.85) and (191.54,87.99) .. (201.53,72.57) ;  
%Straight Lines [id:da8123462105414416] 
\draw [color={rgb, 255:red, 189; green, 16; blue, 224 }  ,draw opacity=0.3 ] [dash pattern={on 4.5pt off 4.5pt}]  (179.17,119.5) -- (228.84,119.97) ;
%Straight Lines [id:da15584489390249834] 
\draw  [dash pattern={on 4.5pt off 4.5pt}]  (179.06,33.06) -- (179.72,269.83) ;
\draw [shift={(179.06,30.06)}, rotate = 89.84] [fill={rgb, 255:red, 0; green, 0; blue, 0 }  ][line width=0.08]  [draw opacity=0] (8.93,-4.29) -- (0,0) -- (8.93,4.29) -- cycle    ;
%Shape: Free Drawing [id:dp9320259320944777] 
\draw  [color={rgb, 255:red, 255; green, 255; blue, 255 }  ,draw opacity=1 ][line width=0.75] [line join = round][line cap = round] (62.83,136) .. controls (62.83,136.08) and (62.83,136.17) .. (62.83,136.25) ;
%Shape: Free Drawing [id:dp41257599185278915] 
\draw  [color={rgb, 255:red, 255; green, 255; blue, 255 }  ,draw opacity=1 ][line width=0.75] [line join = round][line cap = round] (180.33,58.08) .. controls (180.33,58.5) and (180.33,58.92) .. (180.33,59.33) ;
%Shape: Free Drawing [id:dp7448539657974544] 
\draw  [color={rgb, 255:red, 255; green, 255; blue, 255 }  ,draw opacity=1 ][line width=0.75] [line join = round][line cap = round] (180.83,59.08) .. controls (180.58,59.17) and (180.33,59.25) .. (180.08,59.33) ;
%Shape: Circle [id:dp9223778831770422] 
\draw   (89.78,149.94) .. controls (89.78,100.45) and (129.9,60.33) .. (179.39,60.33) .. controls (228.88,60.33) and (269,100.45) .. (269,149.94) .. controls (269,199.44) and (228.88,239.56) .. (179.39,239.56) .. controls (129.9,239.56) and (89.78,199.44) .. (89.78,149.94) -- cycle ;

% Text Node
\draw (186.92,22.08) node [anchor=north west][inner sep=0.75pt]  [font=\small] [align=left] {$\displaystyle x_{4}$};
% Text Node
\draw (137,34.67) node [anchor=north west][inner sep=0.75pt]   [align=left] {$\displaystyle B_{1}^{3}$};
% Text Node
\draw (192,241.33) node [anchor=north west][inner sep=0.75pt]   [align=left] {$\displaystyle B_{2}^{3}$};
% Text Node
\draw (278.5,139.17) node [anchor=north west][inner sep=0.75pt]   [align=left] {$\displaystyle S^{3}$};
% Text Node
\draw (50,229.67) node [anchor=north west][inner sep=0.75pt]   [align=left] {$\displaystyle S^{2} =\ \partial B_{2}^{3}$};
% Text Node
\draw (230.5,30.58) node [anchor=north west][inner sep=0.75pt]   [align=left] {$\displaystyle S^{2} \ =\ \partial B_{1}^{3}$};
% Text Node
\draw (205.6,98.57) node [anchor=north west][inner sep=0.75pt]   [align=left] {$\displaystyle Q$};
% Text Node
\draw (136.5,110.5) node [anchor=north west][inner sep=0.75pt]  [font=\small] [align=left] {$\displaystyle x_{4}( Q)$};

\end{tikzpicture}

    \caption{The  three dimensional sphere $S^3$ with two three dimensional balls $B_1^3$ and $B_2^3$ cut out. After the identification of their boundaries, $S^3\backslash{B_1^3\cup B_2^3}$ becomes $S^2\times S^1$. }
    \label{fig:S3}
\end{figure}
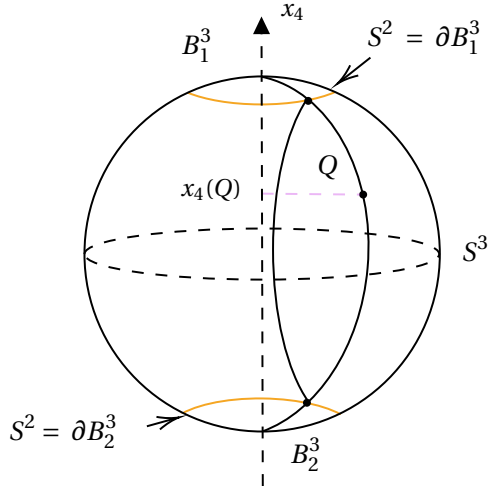
If instead of two lunettes in $\overline{\Delta}$ we have four that are removed, we have correspondingly four balls $B^3$ removed from $S^3.$ We argue that the manifold obtained by removing these four balls from $S^3$ and identifying the four boundary spheres $\partial B^3$ pairwise is homeomorphic to the connected sum of two copies of $S^1\times S^2$. Moving these balls on $S^3$ it is always possible to arrange them in such a way that one pair of the boundary spheres $\partial B^3$ that are identified lie on one of the hemisphere cut out on $S^3$ by a hyperplane containing the $x_4$-axis, while the other pair of boundary spheres lie on the opposite hemisphere. We can cut out $S^3$ along this hyperplane and represent it as a connected sum $S^3\#S^3.$ If we separate each of these spheres, then on each of them we have can remove two balls, identify the corresponding pair of boundaries and obtain by the construction detailed above a manifold homeomorphic to $S^1\times S^2.$ Gluing them again, we get the claim.

We are now in a position to easily deal with the third case. In this case the image of $\pi$ is $S^2$ with four open disks removed, call it $K$. Using stereographic projection with centre inside one of these disks, $K$ can be presented in the plane as in the upper left part of Figure \ref{fig: we cut out the BALLS} ($K$ is the inside of the figure-eight-like configuration). In order to understand the topology of $\pi^{-1}(K)$, we cut $K$ along the segment shown in Figure \ref{fig: we cut out the BALLS} so that $K$ turns out to be homeomorphic to a closed disk $\overline{\Delta}$ in which six small open half-circle $\{C_1, \dots C_6\}$ on the boundary are singled out and identified pairwise (see the lower right part of Figure \ref{fig: we cut out the BALLS} where the half-circles are dashed). Observe that the inverse image $\pi^{-1}(\overline{C_i})$ for each $i$ is indeed homeomorphic to a sphere, since over any point in $C_i$ there is a circle, while over any point in $\partial C_i$ (which is also a boundary point of the disk and a boundary point of the original $K$) there is a point. These spheres are to be identified pairwise to produce $\overline{\Sigma}_{C,h}.$   

Now, instead of singling out the six small half-circles, we add six small lunettes (see the lower left part of figure \ref{fig: we cut out the BALLS}). Each of these lunettes is bounded on one side by one of the $C_i$ and on the other side by piece of the boundary of the disk, in this way obtaining a full closed disk $\overline{\Delta}$. Over each point inside the disk there is circle while over each point in the boundary there is a point. So the inverse image of the disk is indeed $S^3$.
 
The inverse image of each of these lunettes in $S^3$ is again a ball $B^3.$ When we glue $S^3$ pairwise along the boundaries of $\partial B^3$ of these balls (these are just the inverse images of the $\overline{C_i}$'s we singled out before), we get that the resulting manifold is homeomorphic to the connected sum $(S^1\times S^2)\#(S^1\times S^2)\# (S^1\times S^2)$ (one copy of $S^1\times S^2$ for each pair of sphere that is identified).

\end{proof}

 The degenerate case $C=0$ is immediate. In this case indeed, $m_1=m_2=m_3=0$ and the Hamiltonian gives an invariant curve in the $(\xi, p)$ plane given by $\{p^2-\xi=h\}$, that is a family of parabolae, depending on $h$. So the invariant manifold $\Sigma_{0,h}$ is one dimensional and is not compact in this case either. It can be compactified adding a point of at infinity, so  $\overline{\Sigma}_{C,h}$ is homeomorphic to a circle.

\subsection{Bifurcation diagram}\label{bif.sub1}
\begin{figure}
\centering
 \includegraphics[scale=.9]{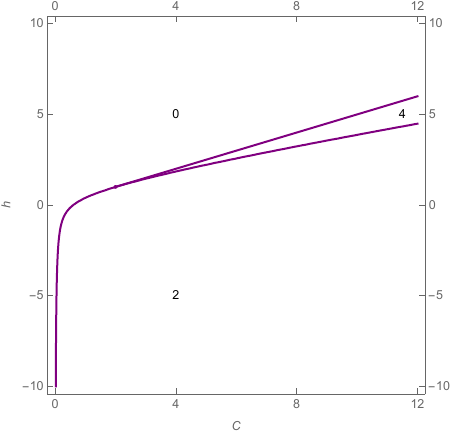}
 \caption{Bifurcation diagram}
  \label{fig: bifurcation}
\end{figure}

Changes in the topology of $\overline{\Sigma}_{C,h}$ depend on the values of $h$ and $C$. In this section, we discuss the corresponding bifurcation diagram. As we explain below, this diagram turns out to be the same as the one appearing in Figure 9 (b) of \cite{borisov2018reduction} describing bifurcation and existence of relative equilibria for different values of $(C, h).$
 
%\[
%y^4 +  y^3 (2C-8 h)+y^2 \left((C-4 h)^2+2\right)+y (8 h-6 C)+1=0
%\]
%\begin{equation}
 %   \begin{split}
       %& y \left(-6 C + 8 h + 2 (2 + (C - 4 h)^2) y_0 + 6 (C - 4 h) y_0^2 + 4 y_0^3)\right)\\&-(y_0 (C-4 h+y_0)+1) (y_0 (C-4 h+3 y_0)-1)
%    \end{split}
%\end{equation}
%$y_0 (C-4 h+3 y_0)-1$ has two solutions:
%\[
%y_0= \frac{1}{6} \left(-\sqrt{(C-4 h)^2+12}-C+4 h\right),\ y_0= \frac{1}{6} \left(\sqrt{(C-4 h)^2+12}-C+4 h\right)
%\]
%One always  positive, one always negative; plug them into 

%The solutions of $y_0 (C-4 h+y_0)+1$ are
%\[
%y_0= \frac{1}{6} \left(-\sqrt{(C-4 h)^2+12}-C+4 h\right),\ y_0= \frac{1}{6} \left(\sqrt{(C-4 h)^2+12}-C+4 h\right)
%\]
%The solutions of $y_0 (C-4 h+y_0)+1$ are

%Both of these solutions are positive when $-c+4h>0$.  

%Pug all the solutions into 

The computations are laborious, so we present just the essential points. 

We consider the polynomial (\ref{eq: bif}) in $m_3$. Since the number of holes in the image of $\pi$ is determined by the number of its roots, the bifurcations happen when this polynomial has zeros of multiplicity greater than one. Since (\ref{eq: bif}) is an even function in $m_3,$ we rewrite it as a polynomial of fourth degree in $y=m_3^2$. We proceed to divide it by monomials $(y-y_0)^2$ and $(y-y_0)^3$ and demand that for each case all the coefficients in the remainders (which are polynomial functions of $y_0$) have a common positive root $y_0^{\ast}$.

For the case of division by $(y-y_0)^3$, computations (which we omit for the sake of brevity) yield only one possible value of the parameters: $(C,h)=(2,1)$. This turns out to be the point where the two curves separate. 

When dividing by $(y-y_0)^2$, we get a remainder of the form 
\begin{equation}
\begin{split}
\label{eq: remainder}
&y \left(2 C^2 y_0-16 C h y_0+6 c y_0^2-6 C+32 h^2 y_0-24 h y_0^2+8 h+4 y_0^3+4 y_0\right)\\&-C^2 y_0^2+8 C h y_0^2-4 C y_0^3-16 h^2 y_0^2+16 h y_0^3-3 y_0^4-2 y_0^2+1
\end{split}
\end{equation}
The constant (in $y$) coefficient simplifies to 
\[
-(y_0 (C-4 h+y_0)+1) (y_0 (C-4 h+3 y_0)-1),
\]
with roots
\begin{equation}
\label{eq: roots y0}
\begin{sqcases}
    y_0= \frac{1}{6} \left(-\sqrt{(C-4 h)^2+12}-C+4 h\right)\\
    y_0 = \frac{1}{6} \left(\sqrt{(C-4 h)^2+12}-C+4 h\right)\\
    y_0=\frac{1}{2} \left(-\sqrt{(C-4 h)^2-4}-C+4 h\right)\\
    y_0= \frac{1}{2} \left(\sqrt{(C-4 h)^2-4}-C+4 h\right)
\end{sqcases}
\end{equation}
The first expression above is always negative; the second always positive; the last two are positive when $4h-C>0$ and negative otherwise. Additionally, they only exist when $(C-4h)^2\ge 4$.

Substituting the first and the last two (assuming they are positive) expressions from (\ref{eq: roots y0}) into the coefficient of $y$ from (\ref{eq: remainder}), taking into account the existence conditions described above and demanding that the coefficient in question be 0, we get the bifurcation diagram in Figure \ref{fig: bifurcation} (the tangent line stops at the point $C=2, h=1$ precisely because $(C-4h)^2\ge 4$).

The two equations in $h$ and $C$ that describe the curves in the bifurcation diagram are as follows:
\begin{equation}
\label{eq: param in c and h}
    \begin{split}
        \begin{sqcases}
            \frac{1}{54} \left(\sqrt{(C-4 h)^2+12}-C+4 h\right)^3+\frac{1}{6} (C-4 h) \left(\sqrt{(C-4 h)^2+12}-c+4 h\right)^2\\+\frac{1}{3} \left((C-4 h)^2+2\right) \left(\sqrt{(C-4 h)^2+12}-C+4 h\right)-6 C+8 h = 0,\\
              h= \frac{C}{2}.
        \end{sqcases}
    \end{split}
\end{equation}
As we have mentioned, this bifurcation diagram seems to  closely resemble  that for relative equilibria in \cite{borisov2018reduction} (Figure 9 (b)). More precisely,
\begin{proposition}
    The two bifurcation diagrams coincide. 
\end{proposition}
\begin{proof}
The bifurcation diagram in \cite{garcia2021attracting} is constructed in the following manner: two types of relative equilibria can be parametrised by $m_3$ and $q$; their explicit forms are
\begin{equation}
    \label{eq: relative equilibria q m3}
\begin{split}
   & \begin{cases}
m_1=p=0,\\
m_3 = \pm\sqrt{\tan\left(\frac{q}{2}\right)}\\
m_2 = \mp\tan\left(\frac{q}{2}\right)^{\frac32}
    \end{cases}, \ 
    \begin{cases}
q = \frac{\pi}{2}\\
m_1=p=0\\
m_2 = -\frac{1}{m_3}.
    \end{cases}
\end{split}
\end{equation}
These expressions, when substituted into the formulae for $\mathcal{C}$ and  $\mathcal{H}$, give two curves in the  $\{C,h\}$ plane, parametrised respectively by $q$ and $m_3$:

    \begin{equation}
    \label{eq: bifurcation q m3}
        \begin{split}
            \begin{sqcases}
            \{C(q),h(q)\} =    \left\{\tan ^3\left(\frac{q}{2}\right)+\tan \left(\frac{q}{2}\right),\frac{1}{2} \left(\tan ^3\left(\frac{q}{2}\right)+\tan \left(\frac{q}{2}\right)\right)+\cot (q) \left(\tan ^2\left(\frac{q}{2}\right)+\tan \left(\frac{q}{2}\right) \cot (q)-1\right)\right\},\\ 
             \{C(m_3),h(m_3)\}   =  \left\{m_3^2+\frac{1}{m_3^2},\frac{1}{2} \left(m_3^2+\frac{1}{m_3^2}\right)\right\}
            \end{sqcases}
        \end{split}
    \end{equation}
On the other hand, equations (\ref{eq: param in c and h}) for the  bifurcations of the topology of $\Sigma_{C,h}$ are formulated explicitly through $C$ and $h$; we want to demonstrate that the loci of (\ref{eq: param in c and h}) and (\ref{eq: bifurcation q m3}) coincide. 

Let us ascertain  that the first equations from the two pairs describe the same curve. To do so, we substitute $C =\tan ^3\left(\frac{q}{2}\right)+\tan \left(\frac{q}{2}\right) $ and  $h =\frac{1}{2} \left(\tan ^3\left(\frac{q}{2}\right)+\tan \left(\frac{q}{2}\right)\right)+\cot (q) \left(\tan ^2\left(\frac{q}{2}\right)+\tan \left(\frac{q}{2}\right) \cot (q)-1\right) $ into the first equation in (\ref{eq: param in c and h}). Simplifying this expression, we get identical 0, which proves that the locus described by \eqref{eq: bifurcation q m3} is contained in one described by \eqref{eq: param in c and h}. To prove the reverse inclusion, it is enough to see that when  $q\in [0, \pi]$, $C$ ranges from $0$ to $+\infty$ and $h$ from $-\infty$ to $+\infty.$

The pair of second equations is simpler: it is clear that they both describe parts of the line $C = 2h$. However, restrictions imposed by square roots in (\ref{eq: roots y0}) stipulate that the part in question is a ray starting at the point $(C,h) = (2,1)$. In order to complete our proof, we need to show that the part of the line described by the second equation in (\ref{eq: relative equilibria q m3}) is the same ray. This can be easily seen by finding a minimum of the function $C(m_3)= m_3^2 + \frac{1}{m_3^2}$; it is achieved when $m_3 = \pm1$ and $C(\pm1) = 2$. 

Thus, the two bifurcation diagrams are the same. 
\end{proof}
\section{Contact structure on $\Sigma_{C,h }$}\label{contact_section}
In this Section, we are going to prove that $\Sigma_{C,h}$ is of contact type when its projection to $S^2_C$ is diffeomorphic to a cylinder and if $h$ is sufficiently negative (see Theorem \ref{th. contact type}).
For the sake of completeness, we first summarize a few results from contact geometry (see \cite{albers2012contact,geiges2006contact} and many others).

%\textcolor{red}{I would shorten this part a bit... since we are not really using too much}
\begin{definition} Let $X$ be a manifold of dimension $2n-1$.
A contact form $\alpha$ on $X$ is a 1-form such that $\alpha\wedge d\alpha^{n-1}$
is a nowhere vanishing volume form on $X$. $(X, \alpha)$ is called a (co-oriented) contact manifold. 
\end{definition}

In contact geometry, one is often interested in a weaker notion, that of a contact structure on $X$: this is just a hyperplane distribution $\eta\subset TX$ which is maximally non-integrable (locally $\eta=\ker(\alpha)$ where $\alpha$ is a (local) 1-form and the condition $\alpha\wedge d\alpha^{n-1}\neq 0$ means that $\eta$ is maximally non-integrable). Notice that if $\alpha$ is a contact form, then $d\alpha_{|\eta}$ is symplectic on $\eta$. For our purposes, we will always use the stronger notion of co-oriented contact manifold $(X, \alpha),$ since this is the notion of contact manifold related to dynamics. 

Indeed, given $(X, \alpha)$ there is associated to it a canonical vector field, called the Reeb vector field $R_{\alpha}$ which is defined via the two conditions $i_{R_{\alpha}}d\alpha=0$ and $i_{R_{\alpha}}\alpha=1$.
Let us remark that the Reeb vector field is always transverse to the contact distribution $\eta:=\ker(\alpha)$ induced by $\alpha,$ since, as we remarked above $d\alpha_{|\eta}$ is symplectic. 

It turns out that in some situations the Reeb vector field is just a (positive) time reparametrization of a Hamiltonian vector field. For this to occur, let $(M, \omega)$ be a symplectic manifold, let $H:M \rightarrow \mathbb{R}$ be a smooth Hamiltonian and let $X:=H^{-1}(e)$, where $e$ is a regular value of $H.$ Assume that there exists a tubular neighborhood $U$ of $X$ in $M$ such that $\omega_{|U}$ is exact, say equal to $d\lambda.$ Thus on $U$ it makes sense to look for a vector field $V$ such that $i_V\omega=\lambda.$ Any such a vector field is called a {\em Liouville} vector field. In particular, if $V$ is a Liouville vector field, it follows automatically that $L_V\omega=\omega$, where $L_V$ is the Lie derivative along $V$. Thus $V$ acts as a symplectic dilation on $U$.

The following important result holds:
\begin{theorem}\label{cont.th.1}
In the situation above, if $V$ is transverse to $X$, then $i_V\omega_{|X}=\lambda_{|X}$ is a contact form $\alpha$ on $X$. Furthermore, the Reeb vector field $R_{\alpha}$ associated to this contact form is a (positive) time reparametrisation of the Hamiltonian vector field $Y_H$
\end{theorem}
For a proof and further details see for instance \cite{frauenfelder2018restricted}. Part of the importance of this result stems from the fact that the Hamiltonian dynamics is described in a completely geometric manner. All the results about the dynamics of the Reeb vector field (like existence of periodic orbits) can be translated into results about Hamiltonian dynamics.

In the situation described by Theorem \ref{cont.th.1}, $(X, \alpha)$ is in particular a hypersurface of contact type inside the symplectic manifold $(M,\omega).$
More in general, following A. Weinstein (see \cite{Weinstein1979}) one has the following:
\begin{definition}\label{def.contact}
    Let $(M, \omega)$ be a symplectic manifold. Let $X\subset M$ be a smooth hypersurface and let $\mathcal{L}_X\subset TX$ be the characteristic line bundle spanned by the kernel of $\omega_{|X}.$ $X$ is called of contact type if there exists a 1-form $\alpha$ on $X$ such that \begin{enumerate}
        \item $d\alpha=\omega_{|X}$ and
        \item $\alpha(v)\neq 0$ for any non-zero $v\in \mathcal{L}_X.$
        \end{enumerate}      
        Then $\alpha$ is called a contact form on $X$.
\end{definition}
In the Definition \ref{def.contact}, $\alpha$ is indeed a contact form on $X$ since $d\alpha_{|\ker(\alpha)}$ is symplectic, so that $\alpha\wedge d\alpha^{n-1}$ defines a nowhere vanishing volume form on $X$. Observe that when $X:=H^{-1}(e)$ for a smooth function $H: M\rightarrow \mathbb{R}$ with regular value $e$, then the associated Hamiltonian vector field $Y_H$ is tangent to $X$ and $Y_H$ is a nowhere vanishing section of $\mathcal{L}_X.$
Thus, as in Theorem \ref{cont.th.1} both $Y_H$ and $R_{\alpha}$ are both nowhere vanishing sections of $\mathcal{L}_X$

%A partial converse of Theorem \ref{cont.th.1} is given by the following (see \cite{moreno2022contact}):
%\begin{theorem}
%Let $(X,\alpha)$ be a contact manifold. Consider the symplectic manifold $\mathbb{R}\times X$ with symplectic form $\omega=d(e^t\alpha)=d\lambda$ with Liouville vector field $V=\partial_t$. Then the embedding $i:X\rightarrow \{1\}\times X \subset \mathbb{R}\times X$
%realizes $(X, \alpha)$ as a hypersurface of contact type inside the symplectic manifold $(\mathbb{R}\times X, \omega).$
%\end{theorem}

Now we proceed to show that if $h$ is sufficienlty negative, $\Sigma_{C,h}$ is a hypersurface of contact type inside the symplectic manifold $(0,\pi)\times\mathbb{R}\times S^2_C$ (in $(q,\ p, \ m_i)$-coordinates) equipped with symplectic form  $$\omega = \mathrm{d}q\wedge \mathrm{d}p + i^*(m_1\mathrm{d}m_2\wedge\mathrm{d}m_3 + m_2\mathrm{d}m_3\wedge\mathrm{d}m_1+ m_3\mathrm{d}m_1\wedge\mathrm{d}m_2),$$
where $i:S^2_C\hookrightarrow \mathbb{R}^3$.

In order to account for the fact that $\{m_1,m_2,m_3\}\in S^2_C$, we make a spherical change of coordinates, substituting $m_1\to \sqrt{C} \cos (\theta ) \cos (\phi ),\ m_2\to \sqrt{C} \cos(\theta ) \sin (\phi),\ m_3\to \sqrt{C} \sin (\theta )$, with $\phi\in[0,2\pi)$ and $\theta\in\left[-\frac{\pi}{2},\frac{\pi}{2}\right]$. 
Thus the symplectic form on $(0,\pi)\times \mathbb{R}\times S^2_C$ written in local coordinates becomes: 
\begin{equation}
\label{eq: omega}
\omega = \mathrm{d}p\wedge \mathrm{d}q  + C^{\frac32}
\cos\left(\theta\right)\mathrm{d}\phi\wedge\mathrm{d}\theta. 
\end{equation}

In Section \ref{sec: topology} we state the inequality condition at a point on $S^2_C$ for preimage of it in $\Sigma_{C,h}$ not to be empty; here, we restate it in terms of $\theta$ and $\phi$:
\begin{equation}
    \label{eq: theta phi preimage condition}
   C \cos ^2(\theta )+\frac{\csc ^2(\theta )}{C}-2 C+2 \cot (\theta ) \sin (\phi )+4 h\ge 0.
\end{equation}
The Hamiltonian will have the form 
\begin{equation}
    \label{eq: Ham  in spherical}
   C \sin ^2(\theta ) \cot ^2(q)+p^2-\cot (q)  -\cos (\theta ) \left(\sqrt{C} p \cos (\phi )+C\sin (\theta ) \cot (q) \sin (\phi )\right)= h-\frac C2
\end{equation}

The  cylindrical subset of $S^2_C$, as the one described by the equation (\ref{eq: theta phi preimage condition}), unlike the entire $S^2_C$, admits a Liouville vector field. Since $\omega$ has two 'independent' parts (see equation \eqref{eq: omega}), so will the Liouville vector field. 
\begin{rmk}
\label{st: theta near 0}
    As $h\to-\infty$, the width of the projection $\Sigma_{C,h}\to S^2_C$ tends to 0. Namely, $\theta$ is contained within a very narrow neighbourhood of 0. We refer to these values of $\theta$ as \textit{permitted values}.
\end{rmk}

\begin{theorem}\label{th. contact type}
    The manifold $\Sigma_{C,h}$ is of contact type when $h = h(C)$ is sufficiently negative. 
\end{theorem}
\begin{proof}

Firstly, observe that  per the bifurcation diagram \ref{fig: bifurcation}, for every value of $C$ we can find $h(C)$ that is negative enough, such that $\pi(\Sigma_{h,C})$ is the complement of two open disks on $S^2$. Secondly, by Remark \eqref{st: theta near 0}  this complement can be made as narrow in $\theta$ as desired by decreasing the value of $h$.  

The proof is based on constructing a Liouville vector field $X$ in a tubular neighborhood of $\Sigma_{C,h}$ inside the symplectic manifold $(0,\pi)\times \mathbb{R}\times S^2_C$.

We divide it in four parts. In the first part, we construct the Liouville vector $X$, but do not completely determine it.

In the second part, we consider points $R$ in the interior of $\pi(\Sigma_{C,h})\subset S^2_C$ and for which $\theta\neq 0$ (points that are not on the equator of $S^2_C$). Then $\pi^{-1}(R)$ consists of a simple closed curve in the $(q.p)$ plane like the one in Figure \ref{fig: pq plane}. In this step we prove that such curves are always transverse in the $(q,p)$ plane to a central vector field  with  a centre in a certain point, whose coordinates depend on $\theta$ and $\phi$. This in turns allows us to determine uniquely the Liouville vector field $X$ so that $X$ is transverse to $\Sigma_{C,h}$ at points whose projection on $S^2_C$ lies in the interior of $\pi(\Sigma_{C,h})$ and does not lie on the equator of $S^2_C.$ 

The third step deals with showing that the Liouville vector field thus obtained is transverse to $\Sigma_{C,h}$ also at points whose projection on $S^2_C$ lies on the equator, i.e. for $\theta=0.$

The last step is focused on showing that the Liouville vector fields is transverse to $\Sigma_{C,h}$ at those points that project to the boundary of $\pi(\Sigma_{C,h})$.

{\bf Step 1} 
For the Liouville vector field $X$ we consider the following ansatz:
\begin{equation}\begin{split}
    \label{eq: first Liouv}
    X  = \frac12\left(\left((p-f_1(\phi,\theta)\right)\frac{\partial}{\partial p} + \left(q - f_2(\phi,\theta)\right)\frac{\partial}{\partial q} + f_3(p,q,\phi,\theta)\frac{\partial}{\partial \phi} + f_4(q,p,\phi,\theta)\frac{\partial}{\partial \theta}\right).
    \end{split}
\end{equation}
%Using Cartan's formula and taking into account that $\omega$ is exact, we get $\mathcal{L}_X\omega = \mathrm{d}(\i_x \omega)$. Explicitly, that is 
%\begin{small}
%\begin{equation*}
%\begin{aligned}
% &
  %\frac12\mathrm{d}\left(\left(f_2-q\right)\mathrm{d}p + \left(p-f_1\right)\mathrm{d}q - C^{\frac32}\cos(\theta)f_4\mathrm{d}\phi + C^{\frac32}\cos(\theta)f_3\mathrm{d}\theta\right) = \frac12\Biggl[\frac{\partial}{\partial \phi}f_2\ \mathrm{d}\phi\wedge\mathrm{d}p  + \partt{\theta}{f_2} \dd{\theta}{p}-\\& \dd{q}{p} + \dd{p}{q} -\partt{\phi}{f_1}\dd{\phi}{q} -\partt{\theta}{f_1}\dd{\theta}{q}  -C^{\frac32}\cos(\theta)\partt{p}{f_4}\dd{p}{\phi} -C^{\frac32}\cos(\theta)\partt{q}{f_4}\dd{q}{\phi} -\\& C^{\frac32}\cos(\theta)\partt{\theta}{f_4}\dd{\theta}{\phi} + C^{\frac32}\sin(\theta)f_4\dd{\theta}{\phi} + C^{\frac32}\cos(\theta)\partt{p}{f_3}\dd{p}{\theta} + C^{\frac32}\cos(\theta)\partt{q}{f_3}\dd{q}{\theta} +\\& C^{\frac32}\cos(\theta)\partt{\phi}{f_3}\dd{\phi}{\theta}\Biggr]. 
%\end{aligned}
%\end{equation*}
%\end{small}
\begin{figure}
    \centering
 \subfigure[$\theta = -1/3,\ h=-10, \ C = \frac12,\ \phi = \frac{\pi}{4}$]{\includegraphics[scale = .7]{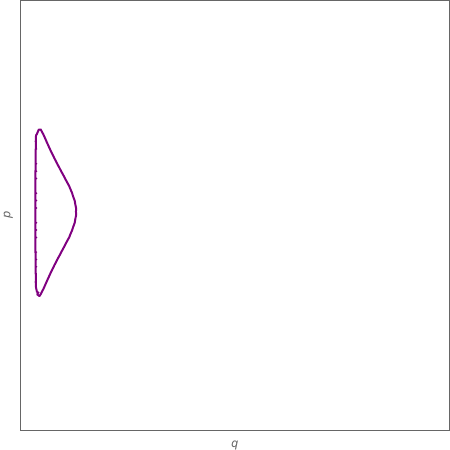}}
    \subfigure[$\theta = 0,\ h=-10, \ C = \frac12,\  \phi = \frac{\pi}{4}$]{\includegraphics[scale = .7]{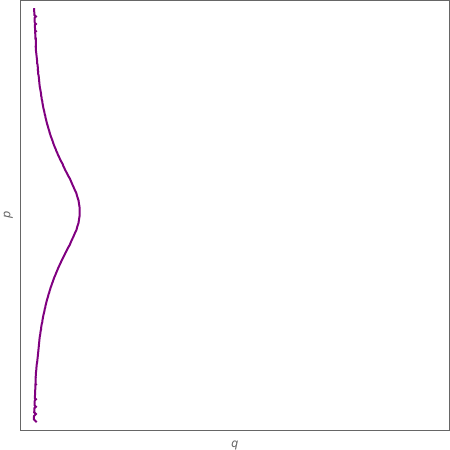}}
   \caption{Intersection of $\Sigma_{C,h}$ with $(q,p)$-plane for various values of $\theta$}
   \label{fig: pq plane}
\end{figure}

Bringing together the like terms and imposing that $X$ in \eqref{eq: first Liouv} is a Liouville vector field for the symplectic form $\omega$ in \eqref{eq: omega}, i.e. $L_X\omega=\omega$ we obtain:
\begin{small}
\begin{equation}
\begin{aligned}
    &\frac12\Biggl[\left(\partt{\phi}{f_2} + C^{\frac32}\cos(\theta)\partt{p}{f_4}\right)\dd{\phi}{p} + \left(\partt{\theta}{f_2}-C^{\frac32}\cos(\theta)\partt{p}{f_3}\right)\dd{\theta}{p} + 2\dd{p}{q}\\&+  \left(-\partt{\phi}{f_1} +C^{\frac32}\cos(\theta)\partt{q}{f_4}\right)\dd{\phi}{q}+\left(-\partt{\theta}{f_1}- C^{\frac32}\cos(\theta)\partt{q}{f_3}\right)\dd{\theta}{q} \\&+ C^{\frac32}\left(\cos(\theta)\partt{\phi}{f_3} + \cos(\theta)\partt{\theta}{f_4} - \sin(\theta)f_4\right)\dd{\phi}{\theta}\Biggr] =\omega.
\end{aligned}
\end{equation}
\end{small}
This entails the following system of partial differential equations:
\begin{equation}
    \label{eq: sys for omega}
\begin{cases}
\cos(\theta)\partt{\phi}{f_3} + \cos(\theta)\partt{\theta}{f_4} - \sin(\theta)f_4 =2\cos(\theta),\\
  \partt{\phi}{f_2} + C^{\frac32}\cos(\theta)\partt{p}{f_4}=0,\\
    \partt{\theta}{f_2}-C^{\frac32}\cos(\theta)\partt{p}{f_3}=0,\\
   \partt{\phi}{f_1} -C^{\frac32}\cos(\theta)\partt{q}{f_4}=0,\\
   \partt{\theta}{f_1}+ C^{\frac32}\cos(\theta)\partt{q}{f_3} = 0.
\end{cases}
\end{equation}
The general solution to this system has the form
\begin{figure}
    \centering
    \includegraphics[scale=.9]{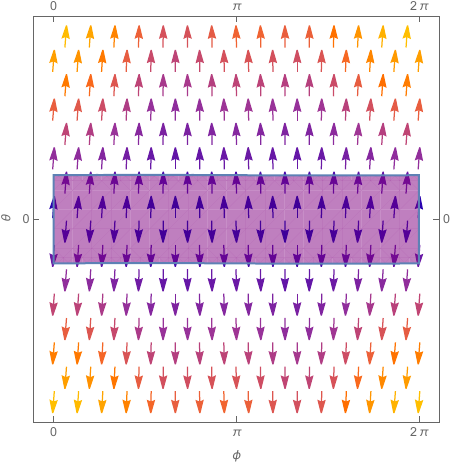}
    \caption{The projection of $\Sigma_{C,h}$ to $S^2$ (in purple) and projection of $X$ with $p$ and $q$ as in (\ref{eq: boundary p and q})}
    \label{fig: proj sigma s2}
\end{figure}
\begin{equation}
\label{eq: sol to the system}
\begin{cases}
    f_3 =\frac{1}{C^{\frac32}\cos(\theta)}\partt{\theta}{f_2}p - \frac{1}{C^{\frac32}\cos(\theta)}\partt{\theta}{f_1}q + m(\phi,\theta), \\
    f_4 =-\frac{1}{C^{\frac32}\cos(\theta)}\partt{\phi}{f_2}p +\frac{1}{C^{\frac32}\cos(\theta)}\partt{\phi}{f_1}q + l(\phi,\theta) \\
    \cos(\theta)\partt{\theta}{l(\phi,\theta)} + \cos(\theta)\partt{\theta}{m(\phi,\theta)} -\sin(\theta)l(\phi,\theta) = 2\cos(\theta).

\end{cases}
\end{equation}
We will determine explicitly the coefficients of $X$ in Step 2. Notice that it depends on two arbitrary functions $f_1,f_2.$

{\bf Step 2}
For illustrative purposes,  we will refer to Figure \ref{fig:curve inpq}; the actual curve will be more stretched vertically but will in principle be of the same shape. 

The expression \ref{eq: Ham  in spherical} is quadratic in $p$; therefore, we can solve it as a regular quadratic equation, obtaining 
\begin{footnotesize}
    \begin{equation}
    \begin{split}
        \label{eq: for p}
p&= \frac{1}{2} \left(\sqrt{C} \cos (\theta ) \cos (\phi )-\sqrt{C\cos ^2(\theta ) \cos ^2(\phi )+4 C \sin (\theta ) \cos (\theta ) \cot (q) \sin (\phi )-4 C \sin ^2(\theta ) \cot ^2(q)-2 C+4 (h+\cot (q))}\right)\\ p&= \frac{1}{2} \left(\sqrt{C} \cos (\theta ) \cos (\phi )+\sqrt{C \cos ^2(\theta ) \cos ^2(\phi )+4 C \sin (\theta ) \cos (\theta ) \cot (q) \sin (\phi )-4 C \sin ^2(\theta ) \cot ^2(q)-2 C+4 (h+\cot (q))}\right)
      \end{split}
    \end{equation}
    \end{footnotesize}
As we remarked above, our shape is symmetric; therefore, we can consider only the expression for the upper part of it, i.e. the second line of (\ref{eq: for p}).

The expression under the square root ,
    \[
C \cos ^2(\theta ) \cos ^2(\phi )+4 C \sin (\theta ) \cos (\theta ) \cot (q) \sin (\phi )-4 C \sin ^2(\theta ) \cot ^2(q)-2 C+4 (h+\cot (q)),
    \]

       can be easily seen to be a quadratic polynomial in $\cot(q)$. 
       
       It has one maximum, denoted in Figure \ref{fig:curve inpq} by  B since its leading coefficient is strictly negative, at least when $\theta\neq0, \pi$ which holds under our assumptions.

       We draw a vertical line $l$ through this point, dividing $\Gamma$ into two parts. Additionally, we will refer to the horizontal line of symmetry as $m$ and their intersection as O. The point where $q$ achieves its minimal value is D and maximal value is C.

\begin{figure}
    \centering
    \includegraphics[scale=.7]{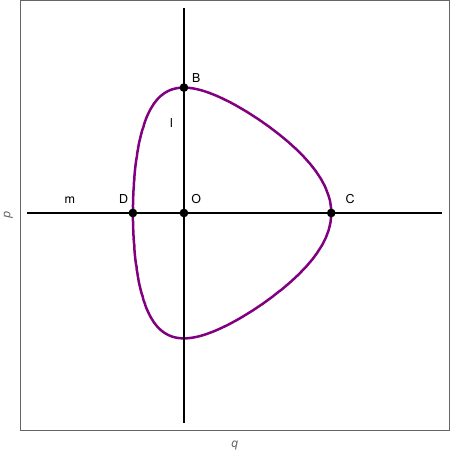}
    \caption{The intersection of $\Sigma_{C,h}$ with the $(p,q)$-plane and its division }
    \label{fig:curve inpq}
\end{figure}

We want to demonstrate the following:
\textbf{any  central vector field centered at O is transverse to $\Gamma$}.
As is was remarked above, the expression under the square root in (\ref{eq: for p}) has one maximum at 
\begin{equation}
 \label{eq: q max}   
q = \arctan\left(\frac{2 C}{C \cot (\theta ) \sin (\phi )+\csc ^2(\theta )}\right).
\end{equation}

Therefore, for every fixed $\theta$ and $\phi$, (\ref{eq: for p}) is a strictly increasing function when $q$ lies between D and O and strictly decreasing otherwise. 

This entails that the tangent vector field to $\Gamma$ to the points "north-east" to the left of  the line l and "south-west" to the right of it.  Any central vector field with centre at O, on the contrary, points "north-west" to the left of l and "north-east" to its left, since O$\in$l. 

Additionally, the tangent vector field to $\Gamma$ is vertical at D and C. This allows us to conclude that such a vector field will be nowhere tangent to $\Gamma$. 

This dictates our choice of $f_1$ and $f_2$: let 
\begin{equation*}
    \begin{split}
        f_1 &:= \frac{\sqrt{C} \cos (\theta ) \cos (\phi )}{2} ,\\
        f_2 &:= \arctan\left(\frac{2 C}{ C\cot (\theta ) \sin (\phi )+\csc ^2(\theta )}\right),
    \end{split}
\end{equation*}

as these are the coordinates of the point O in Figure \ref{fig:curve inpq}. Additionally, we observe that 
\[
\lim\limits_{\theta\to 0}\arctan\left(\frac{2 C}{ C\cot (\theta ) \sin (\phi )+\csc ^2(\theta )}\right) = 0,
\]
so $\lim
\limits_{\theta\rightarrow 0} f_2=0.$

To further simplify $X$ and uniquely fix it, we choose  $m(\phi,\theta) = 0,\ l(\phi,\theta) = 2\tan(\theta)$. Thus, our Liouville vector field of is
\begin{equation}
    \label{eq: final form X}
    \begin{split}
    X &= \frac12\left(p- \frac{\sqrt{C} \cos (\theta ) \cos (\phi )}{2} \right)\partt{p}{} + 
    \frac12\left(q -  \arctan\left(\frac{2 C}{ C\cot (\theta ) \sin (\phi )+\csc ^2(\theta )}\right) \right)\partt{q}{}\\ 
    &+   \left(\frac{2 p \csc ^2(\theta ) \sec (\theta ) (C \sin (\phi )+2 \cot (\theta ))}{\sqrt{C} \left(4 C^2+\left(C \cot (\theta ) \sin (\phi )+\csc ^2(\theta )\right)^2\right)}+\frac{q \tan (\theta ) \cos (\phi )}{2 C}\right)\partt{\phi}{}\\& + \left( \frac{2\sqrt{C} p \csc (\theta ) \cos (\phi )}{4 C^2+\left(C \cot (\theta ) \sin (\phi )+\csc ^2(\theta )\right)^2}+2\tan (\theta )-\frac{q \sin (\phi )}{2 C}\right)\partt{\theta}{}.
    \end{split}
\end{equation}

{\bf Step 3} Recall that the equator in $S^2_C$ is always contained in the interior of $\pi(\Sigma_{C,h})$ and that the curves about it on the $(p,q)$ plane are unbounded. Now we check directly that \textbf{$X$ is transverse to such points in $\Sigma_{C,h}$}. We first observe that:
$\lim\limits_{\theta\to 0} f_3(\theta,\phi,p,q) = 0$.
Now, restricted to the set $\theta = 0$, the Hamiltonian turns into 
    \[
p^2-\sqrt{c} p \cos (\phi )-\cot (q) = h-\frac C2,
    \]
with $X(H)$ being equal to 
\begin{equation}
\label{eq: dir der theta 0}
2p^2-2p\sqrt{C}\cos(\phi)+\frac{C}{2}\cos(\phi)^2+q\csc(q)^2 = \left(\sqrt{2}p 
 - \frac{\sqrt{C}\cos(\phi)}{\sqrt{2}}\right)^2 + q \csc^2(q)>0,
\end{equation}
seeing as $q\in(0,\pi)$.

{\bf Step 4}

Now we show that \textbf{$X$ is transverse to $\Sigma_{C,h}$ at the points that project to the boundary of $\pi(\Sigma_{C,h})$}.
We only need to show that $X$ is transverse to the shape in Figure \ref{fig: proj sigma s2} on boundary points only; this is what we set out to do. 
  
  As remarked above, for all these points on $S^2_C$ the preimage is a single point in $p$ and $q$, namely, 
\begin{equation}
\label{eq: boundary p and q}
\begin{cases}
    p=\frac{\sqrt{C} \cos (\theta ) \cos (\phi )}{2},\\
    q=\arctan\left(\frac{2 C}{C \cot (\theta ) \sin (\phi )+\csc ^2(\theta )}\right)
\end{cases}
\end{equation}
From setting (\ref{eq: theta phi preimage condition})=0, we obtain that
\begin{equation}
\label{eq: boundary sin phi}
\begin{cases}
\sin(\phi) = (C-2 h) \tan (\theta )-\frac{\csc (2 \theta )}{C}-\frac{1}{2} C \sin (\theta ) \cos (\theta ), \\ \cos(\phi)  = \pm\sqrt{1-\left(\frac{1}{4} \tan (\theta ) (C \cos (2 \theta )-3 c+8 h)+\frac{\csc (2 \theta )}{C}\right)^2}
\end{cases}
\end{equation}
nominally giving us two cases for $\cos(\phi)$. 

Tranversality of $X$ with respect to the boundary is checked showing that the scalar product with the normal to the boundary is never vanishing. The (outward pointing) normal in question is given by
\begin{equation}
    \label{eq: normal}
\left\{-2 \cot (\theta ) \cos (\phi ),\frac{2 \csc ^2(\theta ) \left(C \left(C \sin ^3(\theta ) \cos (\theta )+\sin (\phi )\right)+\cot (\theta )\right)}{C}\right\}
\end{equation}

\begin{figure}
    \centering
    \includegraphics[scale=.9]{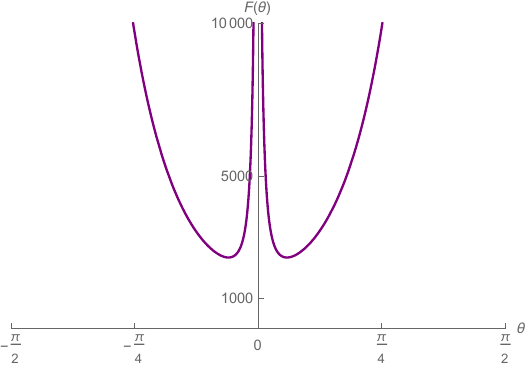}
    \caption{Graph of $F(\theta)$, with $C = 0.64, \ h= -1000$}
    \label{fig: gapf ftheta}
\end{figure}

The next step is taking the scalar product of (\ref{eq: normal}) and (\ref{eq: final form X}), substituting (\ref{eq: boundary p and q}) for $p$ and $q$ and  consequently the two  values of $\sin(\phi)$ and $\cos(\phi)$ from (\ref{eq: boundary sin phi}), all of which yields two functions in $\theta$, which we need to show have constant signs. However, the from of the two functions is identical (this can be verified by observing that the scalar product is quadratic in $\cos(\phi)$. 

The resulting function has the form
\begin{footnotesize}
\begin{equation}
    \label{eq: pokemon final form}
\begin{split}
    F(\theta) &= \frac1C\left(C^2 \sin (2 \theta )-\cot (\theta ) \left(C^2-2 \csc ^2(\theta )\right)+2 C (C-2 h) \csc (\theta ) \sec (\theta )+\csc ^3(\theta ) (-\sec (\theta ))\right)\\& \Biggl[\frac{C \cot (\theta ) \left(1-\left(\frac{1}{4} \tan (\theta ) (C \cos (2 \theta )-3 C+8 h)+\frac{\csc (2 \theta )}{C}\right)^2\right)}{4 C^2+\frac{1}{4} \left(\csc ^2(\theta )-C \left(C \cos ^2(\theta )-2 C+4 h\right)\right)^2}+\\&\frac{\tan (\theta ) \left(C^2 \cos ^2(\theta )-2 C^2+4 Ch+\csc ^2(\theta )\right) \arctan\left(\frac{4 C \csc (\theta )}{C^2 \sin (\theta )+ C^2-4 Ch \csc (\theta )+\csc ^3(\theta )}\right)}{4 C^2}+2 \tan (\theta )\Biggr]-\\&\cot (\theta ) \left(1-\left(\frac{1}{4} \tan (\theta ) (C\cos (2 \theta )-3 C+8 h)+\frac{\csc (2 \theta )}{C}\right)^2\right)\\& \Biggl[\frac{\tan (\theta ) \arctan\left(\frac{4 C \csc (\theta )}{C^2 \sin (\theta )+ C^2-4C h \csc (\theta )+\csc ^3(\theta )}\right)}{C}+\frac{2 \csc ^2(\theta ) \left(-\frac{1}{4} C \tan (\theta ) (C \cos (2 \theta )-3 C+8 h)+2 \cot (\theta )-\csc (2 \theta )\right)}{4 C^2+\frac{1}{4} \left(\csc ^2(\theta )-C \left(C \cos ^2(\theta )-2 C+4 h\right)\right)^2}\Biggr]
\end{split}
\end{equation}
\end{footnotesize}
In order to complete our proof, we need the following
\begin{lemma}
    $F(\theta)>0$ for all permitted values of  $\theta$.
\end{lemma}
\begin{proof}

    The plot of $F(\theta)$ is depicted in Figure \ref{fig: gapf ftheta} for a specific pair $(C,h)$. It clearly is positive and has two minima; however, determining the coordinates of these points analytically is in practice impossible. We will circumvent this by the following: it can be shown that when $\theta\to 0$, 
\begin{equation*}
    \begin{split}
F(\theta) &=\frac{2 C+3}{C^2 \theta ^2} +\left(\frac{1}{C^2}+\frac{36 h-\frac{4}{3}}{C}+2 C-8 h-17\right)\\&
+\frac{\theta ^2 \left(C \left(5 C \left(24 \left(4 C^2-24 C h+C+32 h^2\right)-24 h-127\right)-28\right)+3\right)}{15 C^2} + \overline{O}(\theta^4).
\end{split}
\end{equation*}
This entails that $F(\theta)\to+\infty$ when $\theta\to0$, but also that the speed with which it does so doesn't depend on $h$. However, the width of the strip of permitted values of $\theta$ \textit{does}
depend on $h$ (Remark \ref{st: theta near 0}), and therefore for every value of $C$ we can make $h$ sufficiently negative so that for all values smaller than that $F(\theta)$ will be positive. $F(\theta)$ being positive yields that $X$ is always transverse at those points of $\Sigma_{C,h}$ whose projection is the boundary of of the image of $\pi,$ if the energy is sufficiently negative. 
\end{proof}
 Thus $X$ is everywhere transverse to $\Sigma_{C,h}$ and this proves that $\Sigma_{C,h}$ is a contact type hypersurface.  
\end{proof}


\begin{thebibliography}{10}

\bibitem{albers2012contact}
P.~Albers, U.~Frauenfelder, O.~V. Koert, and G.~P. Paternain.
\newblock Contact geometry of the restricted three-body problem.
\newblock {\em Communications on pure and applied mathematics}, 65(2):229--263,
  2012.

\bibitem{arsie2023collision}
A.~Arsie and N.~A. Balabanova.
\newblock Collision trajectories and regularisation of two-body problem on
  $S^2$.
\newblock {\em Journal of Geometry and Physics}, page 104883, 2023.

\bibitem{bolsinov2004integrable}
A.~Bolsinov and A.~Fomenko.
\newblock {\em Integrable Hamiltonian systems: geometry, topology,
  classification}.
\newblock CRC press, 2004.

\bibitem{borisov2018reduction}
A.~Borisov, L.~Garc{\'\i}a-Naranjo, I.~Mamaev, and J.~Montaldi.
\newblock Reduction and relative equilibria for the two-body problem on spaces
  of constant curvature.
\newblock {\em Celestial Mechanics and Dynamical Astronomy}, 130:1--36, 2018.

\bibitem{fomenko2013algebra}
A.~T. Fomenko and A.~Konyaev.
\newblock Algebra and geometry through hamiltonian systems.
\newblock In {\em Continuous and Distributed Systems: Theory and Applications},
  pages 3--21. Springer, 2013.

\bibitem{frauenfelder2018restricted}
U.~Frauenfelder and O.~Van~Koert.
\newblock {\em The restricted three-body problem and holomorphic curves}.
\newblock Springer, 2018.

\bibitem{garcia2021attracting}
L.~C. Garc{\'\i}a-Naranjo and J.~Montaldi.
\newblock Attracting and repelling 2-body problems on a family of surfaces of
  constant curvature.
\newblock {\em Journal of Dynamics and Differential Equations},
  33(4):1579--1603, 2021.

\bibitem{geiges2006contact}
H.~Geiges.
\newblock Contact geometry.
\newblock In {\em Handbook of differential geometry}, volume~2, pages 315--382.
  Elsevier, 2006.

\bibitem{moreno2022contact}
A.~Moreno.
\newblock Contact geometry in the restricted three-body problem: a survey.
\newblock {\em Journal of Fixed Point Theory and Applications}, 24(2):29, 2022.

\bibitem{ratiu2005crash}
T.~Ratiu.
\newblock {\em A crash course in geometric mechanics}.
\newblock PhD thesis, Quantifization and Harmonic Analysis, 2005.

\bibitem{shchepetilov2006}
A.~V. Shchepetilov.
\newblock Nonintegrability of the two-body problem in constant curvature
  spaces.
\newblock {\em Journal of Physics A: Mathematical and General}, 39:5787--5806,
  2006.

\bibitem{Weinstein1979}
A.~Weinstein.
\newblock On the hypothese of rabinowitz's periodic orbit theorems.
\newblock {\em Journal of Differential Equations}, 39(33):353--358, 1979.

\end{thebibliography}
\end{document}